\documentclass[a4paper,11pt]{article}

\usepackage[T1]{fontenc}
\usepackage[latin1]{inputenc}
\usepackage{amsmath,amsfonts}
\usepackage{amsthm}
\usepackage{amssymb, amsopn}
\usepackage[USenglish]{babel}
\usepackage{amsfonts}
\usepackage{enumerate}
\usepackage{verbatim}
\usepackage{graphicx}
\usepackage{verbatim}
\usepackage[left=3cm, right=3cm, top=3cm, bottom=3cm]{geometry}
\usepackage{tikz}
\usepackage{mathscinet}
\usetikzlibrary{matrix,arrows,decorations.pathmorphing}

\usepackage[final,                 
            pdftex,                 
            pdfpagelabels,
            pdfstartview = {FitH}, 
            bookmarks,              
            colorlinks,            
            plainpages = false,     
            linktoc=all,            
            linkcolor=black,        
            citecolor=blue,       
            urlcolor=blue,         
            filecolor=black,        
            ]{hyperref}

\theoremstyle{plain}
\newtheorem{theorem}{Theorem}
\newtheorem{lemma}[theorem]{Lemma}
\newtheorem{prop}[theorem]{Proposition}
\newtheorem{cor}[theorem]{Corollary}

\theoremstyle{definition}
\newtheorem{definition}{Definition}
\newtheorem{example}{Example}

\theoremstyle{remark}
\newtheorem{remark}{Remark}
\newtheorem{claim}{Claim}

\renewcommand{\tilde}{\widetilde}
\renewcommand{\bar}{\overline}
\newcommand{\bbC}{\mathbb{C}}
\newcommand{\bbZ}{\mathbb{Z}}
\newcommand{\bbN}{\mathbb{N}}
\newcommand{\bbR}{\mathbb{R}}

\newcommand{\raw}{\rightarrow}
\newcommand{\cug}{\subseteq}
\newcommand{\into}{\longmapsto}

\newcommand{\frgl}{\mathfrak{gl}}
\newcommand{\frsl}{\mathfrak{sl}}
\newcommand{\frg}{\mathfrak{g}}
\newcommand{\hgot}{\mathfrak{h}}
\newcommand{\hgotd}{\mathfrak{h}^*_\bbR}
\newcommand{\hgotdc}{\mathfrak{h}^*}
\newcommand{\calX}{\mathcal{X}}

\newcommand{\bbLL}{\mathbb{LL}}

\newcommand{\undw}{{\underline{w}}}
\newcommand{\flipLL}{\raisebox{\depth}{\scalebox{1}[-1]{$LL$}}}

\newcommand{\Address}{
  \bigskip{\footnotesize

  \textsc{Max-Planck-Institut f\"ur Mathematik, Bonn, Germany}\par\nopagebreak
  \textit{E-mail address}: \texttt{leonardo@mpim-bonn.mpg.de}
}}

\DeclareMathOperator{\sym}{Sym}
\DeclareMathOperator{\Trace}{Tr}
\DeclareMathOperator{\Grad}{Gr}
\DeclareMathOperator{\End}{End}

\DeclareMathOperator{\Hom}{Hom}
\DeclareMathOperator{\Ker}{Ker}
\DeclareMathOperator{\ima}{Im}
\DeclareMathOperator{\Span}{span}
\DeclareMathOperator{\Iden}{Id}
\DeclareMathOperator{\defect}{def}

\DeclareMathOperator{\Downs}{downs}

\begin{document}

\title{The N\'eron-Severi Lie algebra of a Soergel module}
\author{Leonardo Patimo}

\maketitle

\begin{abstract}
We introduce the N\'eron-Severi Lie algebra of a Soergel module and we determine it for a large class of Schubert varieties. 
This is achieved by investigating which Soergel modules admit a tensor decomposition.
We also use the N\'eron-Severi Lie algebra to provide an easy proof of  the well-known fact that a Schubert variety is rationally smooth if and only if its Betti numbers satisfy Poincar\'e duality.
\end{abstract}
%%   We insist on abstract in your paper!

\section*{Introduction}
Let $X$ be a smooth complex projective variety of dimension $n$ and $\rho\in H^2(X,\bbR)$ be the Chern class of an ample line bundle on $X$. The Hard Lefschetz Theorem states that for any $k\in \bbN$ cupping with $\rho^k$ yields an isomorphism 
$\rho^k:H^{n-k}(X,\bbR)\raw H^{n+k}(X,\bbR)$. This assures the existence of an adjoint operator $f_\rho\in \frgl(H^*(X,\bbR))$ of degree $-2$ which together with $\rho$ generates a Lie algebra $\frg_\rho$ isomorphic to $\frsl_2(\bbR)$. 
In \cite{LL} Looijenga and Lunts defined the N\'eron-Severi Lie algebra $\frg_{NS}(X)$ of $X$ to be the Lie algebra generated by all the $\frg_\rho$ with $\rho$ an ample class.

The decomposition of $H(X):=H^*(X,\bbR)$ into irreducible $\frg_\rho$-modules is called the primitive decomposition. 
The primitive part (i.e. the lowest weight spaces for the $\frg_\rho$-action) inherits a Hodge structure from the Hodge structure of $H(X)$ and the Hodge structure of the primitive part 
determines completely the Hodge structure on $H(X)$. However, this decomposition depends on the choice of the ample class $\rho$.
Looijenga and Lunts' initial motivation was to find a ``universal'' primitive decomposition of $H(X)$, not depending on any choice: this is achieved by considering the decomposition of $H(X)$ into irreducible $\frg_{NS}(X)$-modules,
which always exists as one can prove that $\frg_{NS}(X)$ is semisimple.
One can easily generalize this construction to any complex variety, possibly singular, by replacing the cohomology $H(X)$ with the intersection cohomology $IH(X)$.

The category of Soergel modules of a Coxeter group $W$ is a full subcategory of the category of graded $R$-modules, where $R$ is a polynomial ring. 
Over a field of characteristic $0$ the category of Soergel modules is a Krull-Schmidt category whose indecomposable objects (up to shifts) are denoted by $\{\bar{B_w}\}_{w\in W}$. 
When $W$ is a Weyl group (of a reductive group $G$) then $\bar{B_{w^{-1}}}\cong IH(X_w)$, where $X_w$ is the Schubert variety corresponding to $w$ inside the flag variety of $G$.

For any real Soergel module $\bar{B_w}$ one can still define its N\'eron-Severi Lie algebra $\frg_{NS}(w)$. Since $\frg_{NS}(w)$ is semisimple and $\bar{B_w}$ is indecomposable as $R$-module (hence as $\frg_{NS}(w)$-module), it follows that $\bar{B_w}$ is an irreducible $\frg_{NS}(w)$-module. 
From this we deduce, in \S 2, an easy proof of the Carrell-Peterson criterion \cite{Ca}: a Schubert variety $X_w$ is rationally smooth if and only if the Poincar\'e polynomial of $H(X_w)$ is symmetric.
In the Appendix we explain how to extend this proof in the setting of a general Coxeter group. 

Looijenga and Lunts went on to compute $\frg_{NS}(X)$ for a flag variety $X=G/B$. They prove that it is ``as big as possible,'' meaning that it is the complete Lie algebra of endomorphisms of $H(X)$ preserving a non-degenerate
(either symmetric or antisymmetric depending on the parity of $\dim X$) bilinear form on $H(X)$.

In \S 3 we explore the case of the N\'eron-Severi Lie algebra $\frg_{NS}(X_w)$ of the intersection cohomology of an arbitrary Schubert variety, a question also posed in \cite{LL}.  
In Proposition \ref{3.7} we show, using a result of Dynkin, that $\frg_{NS}(X_w)$ is maximal if and only if it is a simple Lie algebra. 
If $\frg_{NS}(X_w)$ is not simple then $IH(X_w)$ admits a tensor decomposition $IH(X_w)=A_1\otimes A_2$, where $A_1$ (resp. $A_2$) is a $R_1$ (resp. $R_2$) module and $R_1$, $R_2$ are polynomial algebras with $R=R_1\otimes R_2$. 

Finally in \S4 we try to characterize for which $w\in W$ there is  such a decomposition.
To an element $w\in W$ we associate a graph $\mathcal{I}_w$ whose vertices are the simple reflections $S$, and in which there is an arrow $s\raw t$ whenever $ts\leq w$ and $ts\neq st$. 
We prove that if the graph $\mathcal{I}_w$ is connected and without sinks then a tensor decomposition of $IH(X_w)$ cannot exist,  hence we deduce that in this case $\frg_{NS}(X_w)$ is maximal. 
Thus for the vast majority of Schubert varieties the N\'eron-Severi Lie algebra is ``as big as possible.''

\subsection*{Acknowledgements}
I wish to warmly thank my Ph.D. supervisor Geordie Williamson for introducing me to this problem, and for many useful comment and discussion. I am also grateful to him for explaining me the content of \S A.1.
I would also like to thank
the referee for a careful reading and many useful comments.

Some of this work was completed during a research stay at the RIMS in Kyoto.  
I was supported by the Max Planck Institute in Mathematics. 
\subsection*{Notation}

All cohomology and intersection cohomology groups in this paper are considered with coefficients in the real numbers, unless otherwise stated.
Given a graded vector space or module $M=\bigoplus_{i\in \bbZ} M^i$ we denote by $M[n]$, for $n\in \bbZ$, the shifted module with $M[n]^i=M^{n+i}$.

%%   Use unstarred sectioning commands to obtain automatic numeration.
%% And don't type dots at the end of a heads.
%%   NEW! The number of accessible sectioning commands was highly restricted.
%% Here is the whole list: \section \subsection \subsubsection. This keeps
%% you from making the structure of text too complex.

\section{Lefschetz modules}

In this section we recall from \cite{LL} the definition and the main properties of the N\'eron-Severi Lie algebra. 

Let $M=\bigoplus_{k\in \bbZ} M_k$ be a $\bbZ$-graded finite dimensional $\bbR$-vector space. We denote by $h:M\raw M$ the map which is multiplication by $k$ on $M_k$.
Let $e:M\raw M$ be a linear map of degree $2$ (i.e. $e(M_k)\cug M_{k+2}$ for any $k\in \bbZ$). We say that $e$ has the \textit{Lefschetz property} if for any positive  integer $k$, $e^k$ gives an isomorphism
between $M_{-k}$ and $M_k$.
The Lefschetz property implies the existence of a unique linear map $f:M\raw M$, of degree $-2$, such that $\{e,h,f\}$ is a $\mathfrak{sl}_2$-triple, i.e. $\{e,h,f\}$ span a Lie subalgebra of 
$\mathfrak{gl}(M)$ isomorphic to $\mathfrak{sl}_2(\bbR)$.
We can explicitly construct $f$ as follows: first we decompose $M=\bigoplus_{k\geq 0}\bbR[e](P_{-k})$ where $P_{-k}=\Ker(e^{k+1}|_{M_{-k}})$, 
then we define, for $p_{-k}\in P_{-k}$,
$$f(e^ip_{-k})=\begin{cases} i(k-i+1)e^{i-1}p_{-k} & \text{ if }0<i\leq k, \\ 0 & \text{ if } i=0. \end{cases}$$
The uniqueness of $f$ follows from \cite[Lemma 11.1.1. (VIII)]{Bou}.

\begin{remark}\label{RealCommute}
From the construction of $f$, we also see that if $e$ and $h$ commute with an endomorphism $\varphi \in \frgl(M)$, then $f$ also commutes with $\varphi$.
\end{remark}

\begin{lemma}\label{commute}
If $h$ and $e$ belong to a semisimple subalgebra $\frg$ of $\frgl(M)$, then also $f\in \frg$. 
\end{lemma}
\begin{proof}
Since $\frg$ is semisimple, the adjoint representation of $\frg$ on $\frgl(M)$  induces a splitting $\frg\oplus \mathfrak{a}$, with $[\frg,\mathfrak{a}]\cug \mathfrak{a}$. If $f=f'+f''$ with $f'\in \frg$ and $f''\in \mathfrak{a}$, then $\{e,h,f'\}$ is also an $\frsl_2$-triple. The uniqueness of $f$ implies $f=f'$, thus $f\in \frg$.
\end{proof}
 
Now let $V$  be a finite dimensional $\bbR$-vector space. We regard it as a graded abelian Lie algebra homogeneous in degree $2$ and we consider a graded Lie algebra homomorphism $\mathfrak{e}:V\raw \frgl(M)$ 
(thus the image $\mathfrak{e}(V)$ consists of commuting linear maps of degree $2$). We say that $M$ is a $V$-\emph{Lefschetz module} if there exists $v\in V$ such that $e_v:=\mathfrak{e}(v)$ has the Lefschetz property.
We denote by $V_\mathcal{L}\cug V$ the subset of elements satisfying the Lefschetz property. If $\mathfrak{e}$ is injective, and we can always assume so by replacing $V$ with $\mathfrak{e}(V)$, then $V_\mathcal{L}$ is Zariski
open in $V$. Thus, if $V_\mathcal{L}\neq \emptyset$ there exists a regular map $\mathfrak{f}:V_\mathcal{L}\raw \mathfrak{gl}(M)$ such that $\{\mathfrak{e}(v),h,\mathfrak{f}(v)\}$ is a $\frsl_2$-triple.

\begin{definition}
Let $M$ be a $V$-Lefschetz module.
We define $\frg(V,M)$ to be the Lie subalgebra of $\frgl(M)$ generated by $\mathfrak{e}(V)$ and $\mathfrak{f}(V_\mathcal{L})$. We call $\frg(V,M)$ the \emph{N\'eron-Severi Lie algebra} of the $V$-Lefschetz module $M$.
\end{definition}

The following simple Lemma is needed in Section \ref{Subalgebra}:
\begin{lemma}\label{double}
Let $M$ be a $V$-Lefschetz module. Then $M\oplus M$ is also a $V$-Lefschetz module with respect to the diagonal action of $V$, and $\frg(V,M)\cong \frg(V,M\oplus M)$.
\end{lemma}
\begin{proof}
For any $x\in \frgl(M)$ let $x\oplus x\in \frgl(M\oplus M)$ denote the endomorphism defined by $(x\oplus x)(\mu,\mu')=(x(\mu),x(\mu'))$ for all $\mu,\mu'\in M$.

An element $e \in \frgl(M)$ has the Lefschetz property on $M$ if and only if $e\oplus e$ has the Lefschetz property on $M\oplus M$. Moreover if $\{e,h,f\}$ is an $\frsl_2$-triple in $\frgl(M)$, then $\{e\oplus e,h \oplus h, f\oplus f\}$ is an $\frsl_2$-triple in $\frgl(M\oplus M)$. Therefore the algebra $\frg(V,M\oplus M)$ is generated by the elements $\mathfrak{e}(v)\oplus\mathfrak{e}(v)$, with $v\in V$, and by  $\mathfrak{f}(v)\oplus \mathfrak{f}(v)$, with $v\in V_{\mathcal{L}}$. It follows that the map $x\mapsto x\oplus x$ induces an isomorphism $\frg(V,M)\cong \frg(V,M\oplus M)$.
\end{proof}
\subsection{Polarization of Lefschetz modules}

Assume that $M$ is evenly (resp. oddly) graded and let $\phi: M\times M\raw \bbR$ be a non-degenerate  symmetric (resp. antisymmetric) form such that
$\phi(M_k,M_l)=0$ unless $k\neq -l$. 

We assume for simplicity $V\cug \frgl(M)$. We say that $V$ \emph{preserves} $\phi$ if  every $v\in V$ leaves $\phi$ infinitesimally invariant: 
$$\phi(v(x),y)+\phi(x,v(y))=0\quad\forall x,y\in M.$$ 
 
Since the Lie algebra $\mathfrak{aut}(M,\phi)$ of endomorphisms preserving $\phi$ is semisimple, if $V$ preserves $\phi$ then we can apply the Jacobson-Morozov theorem to deduce that $\frg(V,M)\cug \mathfrak{aut}(M,\phi)$.

For any operator $e:M\raw M$ of degree $2$ preserving $\phi$ we define a form 
$\langle\cdot,\cdot\rangle_e$ on $M_{-k}$, for $k\geq 0$, by $\langle m,m'\rangle_e=\phi(e^km,m')$. One checks easily that $\langle\cdot,\cdot\rangle_e$ is symmetric.

We say that $e$ is a \emph{polarization} if the symmetric form $\langle\cdot,\cdot\rangle_e$ is definite on the primitive part $P_{-k}=\Ker(e^{k+1})|_{M_{-k}}$.
If there exists a polarization $e\in V$, then we call $(M,\phi)$ a \emph{polarized} $V$-Lefschetz module.

\begin{remark} Each polarization $e$ has the Lefschetz property. The injectivity of $e^k|_{M_{-k}}$ follows easily from the non-degeneracy of $\langle\cdot,\cdot\rangle_e$ on $P_{-k}$.
From the non-degeneracy of $\phi$ we get $\dim M_{-k}=\dim M_k$ for any $k\geq 0$, hence $e^k|_{M_{-k}}$ is also surjective.
\end{remark}

\begin{prop}\label{gss}
Let $(M,\phi)$ be a polarized $V$-Lefschetz module. Then the Lie algebra $\frg(V,M)$ is semisimple.
 \end{prop}
\begin{proof}
Since $\frg(V,M)$ is generated by commutators, it is sufficient to prove it is reductive. This will be done by proving that the natural representation on $M$ is completely reducible.
Let $N\cug M$ be a $\frg(V,M)$-submodule. It suffices to show that the restriction of $\phi$ to $N$ is non-degenerate, so that we can take the $\phi$-orthogonal as a complement of $N$.

Let $e\in V$ be a polarization and let $f$ be such that $\{e,h,f\}$ is a $\mathfrak{sl}_2$-triple. We can decompose $N$ into irreducible $\mathfrak{sl}_2$-modules with respect to this triple. We obtain
$N=\bigoplus_{k\geq 0}\bbR[e]P^N_{-k}$ where $P^N_{-k}=\Ker(e^{k+1}|_{N_{-k}})$.
This decomposition is $\phi$-orthogonal since, if $k>h$, we have
$$\phi(e^ap_{-k},e^{\frac{k+h}{2}-a}p_{-h})=(-1)^a\phi(p_{-k},e^{\frac{k+h}{2}}p_{-h})=0$$
for any $p_{-k}\in P_{-k}$, $p_{-h}\in P_{-h}$ and any integer $a\geq 0$.

We consider now a single summand $\bbR[e]P^N_{-k}$.
Because the form $\langle\cdot,\cdot\rangle_e$ is definite on $P^N_{-k}\cug P_{-k}$, it follows that $\phi$ is non-degenerate on $P^N_{-k}+e^kP^N_{-k}$. Since $e$ preserves $\phi$,  
the restriction of $\phi$ to $e^aP^N_{-k}+e^{k-a}P^N_{-k}$ is also non-degenerate for any $0\leq a\leq k$. We conclude since the subspaces $e^aP^N_{-k}+e^{k-a}P^N_{-k}$ and $e^bP^N_{-k}+e^{k-b}P^N_{-k}$ are $\phi$-orthogonal for $a\neq b,k-b$.
\end{proof}

\begin{remark}
The proof of Proposition \ref{gss} actually shows that the Lie algebra generated by $V$ and $\mathfrak{f}(e)$, where $e$ is a polarization, is semisimple. Therefore, by Lemma \ref{commute}, if $e$ is any polarization in $V$, then $V$ and $\mathfrak{f}(e)$ generate $\frg(V,M)$.
\end{remark}

\begin{cor}\label{cor1}
Let $(M, \phi)$ be a polarized $V$-Lefschetz module. If $N\cug M$ is a graded $V$-submodule satisfying $\dim N_{-k}=\dim N_k$ for any $k\geq 0$, then there exists a complement $N'\cug M$ such that $M=N\oplus N'$ as a $\frg(V,M)$-module.
\end{cor}

\begin{proof}
Let $v\in V$ having the Lefschetz property on $M$. Since $v^k|_{N_{-k}}$ is injective and $\dim N_{-k}=\dim N_k$,  $v$ also has the Lefschetz property on $N$, therefore $N$ is $\mathfrak{f}(v)$-stable. This implies that $N$ is a $\frg(V,M)$-submodule of $M$.

As in the proof of Proposition \ref{gss} one can show that the restriction of $\phi$ to $N$ is non-degenerate, so the $\phi$-orthogonal subspace $N'$ is a $\frg(V,M)$-stable complement of $N$.
\end{proof}

\begin{remark}
The definitions given above arise naturally in the setting of complex projective (or compact K\"ahler) manifolds. Let $X$ be a complex projective manifold of complex dimension $n$ and assume that $X$ is of Hodge-Tate type, 
i.e. if 
$$\displaystyle H^*(X,\bbC)=\bigoplus_{p,q\geq 0} H^{p,q}$$ is the Hodge decomposition of $X$ then $H^{p,q}=0$ for $p\neq q$. In particular the cohomology of $X$ vanishes in odd degrees.

Let $M=H(X,\mathbb{R})[n]$ be the cohomology of $X$ shifted by $n$ and let $\phi$ be the \emph{intersection form}:
$$\phi(\alpha,\beta)=(-1)^\frac{k(k-1)}{2}\int_X\alpha\wedge \beta,\qquad\forall\alpha\in H^{n+k}(X,\mathbb{R}),\;\forall\beta\in H^{n-k}(X,\mathbb{R}).$$
Notice that $\phi$ is symmetric (resp. antisymmetric) if $n$ is even (resp. $n$ is odd).

Let $\rho\in H^2(X,\mathbb{R})$ be the first Chern class of an ample line bundle on $X$. Then the Hard Lefschetz theorem and the Hodge-Riemann bilinear relations imply that $\rho$ is a polarization of $(M,\phi)$. 
It follows that $(M,\phi)$ is a polarized Lefschetz module over $H^2(X,\mathbb{R})$.

We can also replace $H^2(X,\mathbb{R})$ by the \emph{N\'eron-Severi group} $NS(X)$, i.e. the subspace of $H^2(X,\mathbb{R})$ generated by Chern classes of line bundles on $X$. We define the \emph{N\'eron-Severi Lie algebra} of $X$ as $\frg_{NS}(X)=\frg(NS(X),H^\bullet(X,\mathbb{R})[n])$.

In \cite{LL} Looijenga and Lunts consider complex manifolds with an arbitrary Hodge structure. To deal with the general case one needs to modify the definition of polarization given here in order to make it compatible with the general form of the Hodge-Riemann bilinear relations.

However all the Schubert varieties, the case in which we are mostly interested, are of Hodge-Tate type, so for simplicity we can limit ourselves to this case.
\end{remark}

\subsection{Lefschetz modules and weight filtrations}\label{WeightSection}

 Let $V$ be a finite dimensional $\mathbb{R}$-vector space and $(M,\phi)$ a polarized $V$-Lefschetz module.  In this section we show how to each element $v\in V$ we can associate a weight filtration and to any such filtration we can associate a subalgebra of  $\frg(V,M)$. In many situations the knowledge of these subalgebras turns out to be an important tool  to study $\frg(V,M)$.

\begin{lemma}\label{Weight}
Let $e$ be a nilpotent operator acting on a finite dimensional vector space $M$ such that $e^l\neq 0$ and $e^{l+1}=0$. Then there exists a unique non-increasing filtration $W$, called the \textit{weight filtration}.
$$\{0\} \cug W_{l}\cug W_{l-1}\cug \ldots \cug W_{-l+1}\cug W_{-l}=M$$ 
such that
\begin{itemize}
\item $e(W_k)\cug W_{k+2}$ for all $k$;
\item for any $0\leq k\leq l$, $e^k:\Grad^W_{-k}(M)\raw \Grad^W_k(M)$ is an isomorphism, where $\Grad^W_k(M)=W_k/W_{k+1}$.
\end{itemize}
\end{lemma}
\begin{proof}See, for example, \cite[Proposition A.2.2]{CaZe}.\end{proof}

\begin{lemma}\label{JM0}
Let $e\in V$ (not necessarily a Lefschetz operator). Then there exists a $\frsl_2$-triple $\{e, h',f' \}$ contained in $\frg(V,M)$ such that $h'$ is of degree $0$.
\end{lemma}
\begin{proof}
This is \cite[Lemma 5.2]{LL}.
\end{proof}

Let  $\{e,h',f'\}$ be as is Lemma \ref{JM0} and $W_{\bullet}$ be the weight filtration of $e$.
Since $h'$ is semisimple and part of a $\frsl_2$-triple, we have a decomposition in eigenspaces $M=\bigoplus_{n\in \mathbb{Z}} M'_n$, where $M'_n=\{x \in M \mid h'\cdot x=nx\}$. We can define $\tilde{W}_k=\bigoplus_{n\geq k} M'_n$. It is easy to check that $\tilde{W}_\bullet$ satisfies the defining condition of the weight filtration of $e$. In particular, $W_{\bullet}=\tilde{W}_{\bullet}$ and $h'$ splits the weight filtration of $e$, i.e. 
$W_k=W_{k+1}\oplus M'_k$ for all $k$.

Let $h''=h-h'$. Then $(h',h'')$ is a commuting pair of semisimple elements in $\frg(V,M)$ and it defines a bigrading $M^{p,q}$ on $M$ such that $M^n=\bigoplus_{p+q=n} M^{p,q}$. Furthermore $h'$ and $h''$ also act via the adjoint representation on $\frg(V,M)$ defining a bigrading $\frg(V,M)_{p,q}$. We have $x\in \frg(V,M)_{p,q}$ if and only if $x(M^{p',q'})\cug M^{p+p',q+q'}$ for all $p',q'\in \mathbb{Z}$. For $x\in \frg(V,M)$ we denote by $x_{p,q}$ its component in $\frg(V,M)_{p,q}$.

Let $\tilde{V}$ be a subspace of $V$ containing $e$ and such that, for any $x\in \tilde{V}$, we have $x(W_k)\cug W_{k+2}$ for all $k$. Consider the graded vector space  $\Grad^WM=\bigoplus_{k\in \bbZ}\Grad_k^WM$, where $\Grad_k^WM$ sits in degree $k$.
Then $\Grad^WM$ is a $\tilde{V}$-Lefschetz module, so we can define the Lie algebra $\frg(\tilde{V},\Grad^WM)$.

Let $x\in \tilde{V}$. Since $x(W_k)\cug W_{k+2}$, then $x(M'_k)\cug \bigoplus_{n\geq k+2} M'_n$.
This implies that $x\in \frg(V,M)_{\geq 2,\bullet}$, i.e.
$x=x_{2,0}+x_{4,-2}+x_{6,-4}+\ldots$. In particular, if $x,y\in \tilde{V}$, we have $[x,y]=0$ and so $[x_{2,0},y_{2,0}]=[x,y]_{4,0}=0$.

Let $\tilde{V}_{2,0}\cug \frg(V,M)$ be the span of the degree $(2,0)$ components of elements of $\tilde{V}$. The subspace $\tilde{V}_{2,0}$ is an abelian subalgebra of $\frg(V,M)$. However, notice that in general $\tilde{V}_{2,0}$ is not a subspace of $V$.  
We denote by $M'$ the vector space $M$ with the grading defined by $h'$. Then $M'$ is a $\tilde{V}_{2,0}$-Lefschetz module (in fact $e=e_{2,0}$ is a Lefschetz operator on $M'$), so we can define the algebra $\frg(\tilde{V}_{2,0},M')$.

\begin{prop}\label{WeirdGrading}
In the setting as above,
there exists an isomorphism of Lie algebras
$ \frg(\tilde{V},\Grad^WM)\cong\frg(\tilde{V}_{2,0},M')$. In particular $\frg(V,M)$ contains a subalgebra isomorphic to $\frg(\tilde{V},\Grad^WM)$.
\end{prop}
\begin{proof}
Let $\pi_k: W_k\raw M'_k$ be the projection. Then $\bigoplus_k \pi_k:Gr^WM\raw M'$ is an isomorphism of graded vector spaces.

Moreover, the isomorphism $\bigoplus_k \pi_k$ is compatible with the map $\tilde{V}\raw \tilde{V}_{2,0}$ given by $x\mapsto x_{2,0}$, i.e. that for any $k\in \bbZ$ the following diagram commutes: 
\begin{center}
\begin{tikzpicture}
	\matrix (m) [matrix of math nodes,row sep=3em,column sep=2.5em,minimum width=2em,text height=1.5ex, text depth=0.25ex]
  {W_{k+2}/W_{k+3} & M'_{k+2}\\
  W_k/W_{k+1} & M'_k\\};
  \path[->]
  (m-1-1) edge node[above]{$\pi_{k+2}$} (m-1-2)
  (m-2-1) edge node[above]{$\pi_k$} (m-2-2)
  (m-2-1) edge node[left]{$x$}(m-1-1)
  (m-2-2) edge node[right]{$x_{2,0}$} (m-1-2)	  
  ;
\end{tikzpicture}
\end{center}
Hence, it follows that $ \frg(\tilde{V},\Grad^WM)\cong\frg(\tilde{V}_{2,0},M')$.

The last statement follows from Lemma \ref{commute}, in fact both $\tilde{V}_{2,0}$ and $h'$ are contained in $\frg(V,M)$, whence $\frg(\tilde{V}_{2,0},M')\cug \frg(V,M)$.
\end{proof}

\section{Application to Soergel Calculus}\label{SoergelCalc}

Let $G$ be a simply-connected complex reductive Lie group, $B$ be a Borel subgroup of $G$ and $T\cug B$ be a maximal torus. We denote by $X=G/B$ its flag variety. 
Let $\frg$ be the Lie algebra of $G$ and $\hgot\cug\frg $ be the Lie algebra of $T$, with dual space $\hgotdc$.
Let $\Phi\cug \hgotdc$ be the root system of $G$ and $\Delta$ be the set of simple roots with respect to $B$. Let $W$ be the Weyl group of $G$ and $S\cug W$ be the set of simple reflections. 
We denote by $(\cdot,\cdot)$ the Killing form on $\hgotdc$. 

Let $\Lambda=\{\lambda\in \hgotdc\mid 2(\lambda,\alpha)/(\alpha,\alpha)\in \bbZ \; \forall \alpha \in \Phi\}$ be the weight lattice and let $\hgotd=\Lambda\otimes_\bbZ\bbR$. 
For any weight $\lambda\in \Lambda$ we define a one-dimensional module $\bbC_\lambda$ of $B$. Then the projection $G\times_B \bbC_\lambda\raw G/B$ is a line bundle $L_\lambda$ on $X$ and the first Chern class $c_1(L_\lambda)$ defines an element in $H^2(X):=H^2(X,\bbR)$. 
The map $\lambda \into c_1(L_\lambda)$ induces a homomorphism $\Lambda\raw H^2(X)$ which can be extended to a graded algebra homomorphism from $R=\text{Sym}(\hgotd)=\bbR[\hgotd]$ to $H(X)$ where $\hgotd$ is regarded as homogeneous polynomials of degree $2$.

This map is surjective, and its kernel is the ideal generated by $R^W_+$, the invariants in positive degree under the Weyl group $W$ of $G$.

Note that in the Hodge decomposition of $X$ only terms of type $(p,p)$ appear. Furthermore we have $NS(X)=H^2(X)=\hgotd=R^2$, since $(R^W_+)^2=0$.

Let $w\in W$ and let $IH_w:=IH(X_w,\bbR)$ be the intersection cohomology of the Schubert variety $X_w=\bar{B\cdot wB}\overset{i}{\hookrightarrow} X$. We regard $IH_w$ in a natural way as a $R$-module via the composition map $R\raw H(X)\xrightarrow{i^*} H(X_w)$.

\begin{remark}\label{natmap}
For any complex variety $Y$, there is a natural map $H(Y)[\dim Y]\raw IH(Y)$. If $Y$ is projective, then the kernel is precisely the non-pure part of $H(Y)$ \cite[Theorem 3.2.1]{dCM2}.
Because Schubert varieties have a cell decomposition, their cohomology is pure. Hence, we have a natural inclusion $H(X_w)[\ell(w)]\hookrightarrow IH_w$ for any $w\in W$.
\end{remark}

The $R$-modules arising as intersection cohomology of Schubert varieties can also be defined purely algebraically. Let $\underline{w}=s_1s_2\ldots s_{\ell}$ be a reduced expression 
for $w\in W$, where $\ell:=\ell(w)$ is the length of $w$ and $s_i\in S\cug W$.
We define the Bott-Samelson module
$\bar{BS}(\underline{w})=R\otimes_{R^{s_1}}R\otimes_{R^{s_2}}\ldots R\otimes_{R^{s_\ell}}R\otimes_R \bbR[\ell]$. Here
 $R^{s_i}$ denotes the $s_i$-invariants in $R$ and $\bbR$ has the $R$-module structure given by $\bbR\cong R/R_{>0}$.

\begin{theorem}[Soergel, \cite{S1}]
We choose any decomposition of  $\bar{BS}(\underline{w})$ into indecomposable $R$-modules and we denote by $\bar{B_w}$ the summand containing $1^\otimes:=1\otimes \ldots \otimes 1$. Then:
\begin{enumerate}[i)]
\item Up to isomorphism, $\bar{B_w}$ does not depend on the choice of decomposition, nor on the choice of the reduced expression $\underline{w}$ of $w$.
\item Any indecomposable summand of $\bar{BS}(\underline{w})$ is isomorphic, up to shift, to a module $\bar{B_{w'}}$ for some $w'\leq w$.
\item $IH_w\cong \bar{B_{w^{-1}}}$ for any $w\in W$.
\end{enumerate}
\end{theorem}

As Soergel pointed out, the definition of the module $B_w$ can be easily generalized to any Coxeter group $W$ with $\hgotd$ replaced by a reflection faithful representation of $W$ (in the sense of \cite[Definition 1.5]{S4}). 
For a general Coxeter group there are no known varieties such that 
intersection cohomology gives the indecomposable Soergel module. Nevertheless there exists a replacement for the intersection form in this setting. 

The degree $\ell=\ell(w)$ component of $\bar{BS}(\underline{w})$ is one-dimensional and it is spanned by $c_{\text{top}}:=\alpha_{s_1}\otimes\alpha_{s_2}\otimes\ldots\otimes \alpha_{s_\ell}\otimes 1$. 
Here $\alpha_s$ denotes the simple root corresponding with $s\in S$.
We define the intersection form $\phi$ on $\bar{BS}(\underline{w})$ via 
$$\phi(f,g)=(-1)^{\frac{k(k-1)}2}\Trace(fg)\qquad \forall f\in \bar{BS}(\underline{w})^k,\; \forall g\in \bar{BS}(\underline{w})^{-k},\; \forall k\in \bbZ$$
where $fg$ denotes the term-wise multiplication, and $\Trace$ is the functional which returns the coefficient of $c_{\text{top}}$.
The restriction of the intersection form $\phi$ to $\bar{B_w}$ is well-defined up to a positive scalar and it is non-degenerate.

\begin{theorem}[Elias-Williamson \cite{EW1}]
Let $\eta\in \hgotd$ be in the ample cone, i.e $( \eta, \alpha)>0$ for any $\alpha \in \Delta$. 
Then left multiplication by $\eta^r$ induces an isomorphism $\eta^r:\left(\bar{B_w}\right)^{-r}\raw \left(\bar{B_w}\right)^{r}$ for any $r\geq 0$. 

Furthermore, if $\phi$ is the intersection form of $\bar{B_w}$ then we can define a non-degenerate symmetric product $\langle\cdot,\cdot\rangle_\eta$ on $(\bar{B_w})^{-r}$ via 
 $\langle \alpha,\beta\rangle_\eta \cong \phi(\eta^r\alpha,\bar{\beta})$. 
This symmetric product is $(-1)^{\ell(w)(\ell(w)+1)/2}$-definite when restricted to the primitive part $P^{-r}=\Ker\left(\eta^{r+1}|_{(\bar{B_w})^{-r}}\right)$.
\end{theorem}
This means that the N\'eron-Severi Lie algebra can still be defined for Soergel modules as $\frg_{NS}(w):=\frg(\hgotd,\bar{B_w})$.
We can now apply Corollary \ref{cor1} to the polarized $\hgotd$-Lefschetz module $\bar{B_w}$.

\begin{cor}\label{betti}
Let $N$ be a non-zero $R$-submodule of $\bar{B_w}$ such that $\dim N^{-k}=\dim N^k$ for any $k\in \bbZ$. Then $N\cong \bar{B_w}$.
\end{cor}

If $w\in W$ and $s\in S$ such that $ws>w$, then $\displaystyle \bar{B_wB_s}=\bar{B_{ws}}\oplus \bigoplus_{z< ws} \bar{B_z}^{m_z}$ for some $m_z\in \mathbb{Z}_{\geq0}$, cf. \cite[\S 1.2.3]{EW1}. In particular $\bar{B_wB_s}$ is a polarized $\hgotd$-Lefschetz module.

\begin{cor}
Let $N$ be a $R$-submodule of $\bar{B_wB_s}$ such that $\dim N^{-k}=\dim N^k$ for any $k\in \bbZ$. Then $N$ is a direct summand of $\bar{B_wB_s}$. In particular, if $N$ is indecomposable and $N^{-\ell(ws)}\neq 0$, then $N\cong \bar{B_{ws}}$.
\end{cor}

We now restrict ourselves the case where $W$ is the Weyl group of a simply-connected complex reductive group.
We recall some results from \cite{BGG}. The elements $[X_v]\in H_{2\ell(v)}(X)$, the fundamental classes of the Schubert varieties $X_v$ (for $v\in W$), are a basis of the homology of $X$. 
By taking the dual basis we obtain a basis $Q_v\in H^{2\ell(v)}(X)$, for $v\in W$, of the cohomology, called the \textit{Schubert basis}.

Let $i:X_w\hookrightarrow X$ denote the inclusion. Then $i^*: H(X)\raw H(X_w)=:H_w$ is surjective:
$i^*(Q_v)=0$ if and only if $v\not\leq w$ and the set $\{i^*(Q_v)\}_{v\leq w}$ (which we will denote simply by $Q_v$) is a basis of $H_w$. 

The following result is due to Carrell-Peterson \cite{Ca}:

\begin{cor}\label{corPD}
For any $w \in W$ the following are equivalent:
\begin{enumerate}[i)]
 \item $H_w[\ell(w)]=IH_w$.
 \item $\#\{v\in W\mid v\leq w\text{ and }\ell(v)=k\}=\#\{v\in W\mid v\leq w\text{ and }\ell(v)=\ell(w)-k\}$ for any $k\in \bbZ$.
 \item All the Kazhdan-Lusztig polynomials $p_{v,w}$ are trivial.
\end{enumerate}

\end{cor}
\begin{proof}
The shifted cohomology $H_w[\ell(w)]$ is a $R$-submodule of the indecomposable $R$-module $\bar{B_{w^{-1}}}=IH_w$ (Remark \ref{natmap}) and 
$$\dim H^{2k}(X_w)=\#\{v\in W\mid v\leq w\text{ and }\ell(v)=k\}.$$

If $\dim H^{2k}(X_w)=\dim H^{2\ell(w)-2k}(X_w)$ for any $0\leq k\leq \ell(w)$, from Corollary \ref{betti} we get that
$H_w[\ell(w)]$ and $IH_w$ must coincide, thus ii) implies i). 
Vice versa, i) implies ii) because $IH_w$ satisfies $\dim IH_w^{-k}=\dim IH_w^{k}$ for any $k \in \mathbb{Z}$.

Because $p_{v,w}(0)=1$ for any $v\leq w$, we have $\dim H_w=\sum_{v\leq w} p_{v,w}(0)$, while $\dim IH_w=\sum_{v\leq w} p_{v,w}(1)$.
Since the KL polynomials have positive coefficients, we have $\dim IH_w=\dim H_w$ if and only if $p_{v,w}(1)=p_{v,w}(0)$ for any $v\in W$, or equivalently if and only if $p_{v,w}(q)=1$ for any $v\leq w$. It follows that i) is equivalent to iii).
\end{proof}

A similar argument works also for a general Coxeter group $W$. We explain in the \hyperref[appendix]{Appendix} how to extend the proof of Corollary \ref{corPD} to that setting.

\section{The N\'eron-Severi algebra of Soergel modules}

In \cite{LL} Looijenga and Lunts determined the N\'eron-Severi Lie algebra $\frg_{NS}(X)$ of a flag variety $X=G/B$ of every simple group $G$: it is the complete algebra of automorphisms $\mathfrak(H,\phi)$ of the intersection form,
i.e. it is a symplectic (resp. orthogonal) algebra if the complex dimension of $X$ is odd (resp. even).
%(so it is the symplectic algebra for $A_n$, for $n\equiv 1,2\mod 4$, $B_n$ and $C_n$, for $n$ odd, and $E_7$, and the orthogonal algebra in the remaining cases).

Here we want to extend their results and determine the Lie algebra $\frg_{NS}(w)$ for an arbitrary $w\in W$.
We don't quite succeed, however we show that $\frg_{NS}(w)$ is ``as large as possible'' for many $w$.

\subsection{Basic properties of the Schubert basis}
Let $\{Q_v\}_{v\in W}$ be the Schubert basis of $H(X)$ introduced in Section \ref{SoergelCalc}.
The $R$-module structure of $H(X)$ can be described in the basis $\{Q_v\}_{v\in W}$ by the Chevalley formula \cite[Theorem 3.14]{BGG}:
\begin{equation}\label{cheva}
\lambda\cdot Q_w=2\sum_{w\xrightarrow{\gamma} v}\frac{(w\lambda,\gamma)}{(\gamma,\gamma)}Q_v
\end{equation}
where the notation $w\xrightarrow{\gamma} v$ means $\ell(v)=\ell(w)+1$, $\gamma\in \Phi^+$ and $v=s_\gamma w$, where $s_\gamma\in W$ is the reflection corresponding to $\gamma$.
 
In particular if $s\in S$ then $Q_s\in H^2(X)=\hgotd$ can be identified with the fundamental weight in $\Lambda$ corresponding to $\alpha_s$, i.e. we have
$2(Q_s,\alpha_s)=(\alpha_s,\alpha_s)$ and 
$(Q_s,\alpha_t)=0$ for any $s\neq t\in S$.
The following Lemma is an easy application of the Chevalley formula \eqref{cheva}:

\begin{lemma}\label{conti}
In $H(X)$ we have, for any $s,t\in S$:
\begin{enumerate}[i)]
\item $\displaystyle Q_s^2=-2\sum_{u\in S\setminus\{s\}}\frac{( \alpha_s,\alpha_u)}{(\alpha_u,\alpha_u)} Q_{us}$;

\item $Q_sQ_t=Q_{st}$ if $(\alpha_s,\alpha_t)= 0$;
\item $Q_sQ_t=Q_{st}+Q_{ts}$ if $(\alpha_s,\alpha_t)\neq 0$ and $s\neq t$.
\end{enumerate}
\end{lemma}

We state here for later reference a preliminary lemma.

\begin{lemma}\label{kilinv}
 If the root system $\Phi$ is irreducible (i.e. if the Dynkin diagram of $G$ is connected) then $(R^W_+)^4\cong \bbR$ and it is spanned by 
 $$\calX=\sum_{s,t\in S}c_{st}Q_sQ_t\qquad \text{ where }\quad c_{st}=\frac{(\alpha_s,\alpha_t)}{(\alpha_s,\alpha_s)(\alpha_t,\alpha_t)}.$$
\end{lemma}
\begin{proof}
 A $W$-invariant element in $R^4=\sym^2(\hgotd)$ corresponds to a $W$-equivariant morphism $\hgot_\bbR\raw \hgotd$, where $\hgot_\bbR=\{x\in \hgot \mid \lambda(x)\in \bbR \;\forall \lambda \in \hgotd\}$. 
 Since $\hgot_\bbR$ and $\hgotd$ are irreducible as $W$-modules, such a morphism is unique up to a 
 scalar. 
 The Killing form $(\cdot,\cdot)$ is $W$-invariant, hence $\eta\mapsto (\eta,\cdot)$ is a $W$-equivariant isomorphism $\hgotd\raw \hgot_\bbR$. 
 
 For any $x\in \hgotd$ we have $x=\sum_{s\in S}\frac{2(Q_s,x)}{(\alpha_s,\alpha_s)}\alpha_s$. Hence for any $x,y\in \hgotd$ 
 $$(x,y)=\sum_{s,t\in S}\frac{4(\alpha_s,\alpha_t)}{(\alpha_s,\alpha_s)(\alpha_t,\alpha_t)}(Q_s,x)(Q_t,x)=4\sum_{s,t\in S}c_{st}(Q_s,x)(Q_t,y)$$
whence $\sum_{s,t\in S}c_{st}Q_sQ_t\in R^4$ is $W$-invariant.
\end{proof}

\begin{remark}\label{nondeg}
The element $\calX$ is basically (up to a scalar) just the Killing form written in the basis $\{Q_sQ_t\}_{s,t\in S}$ of $\sym^2(\hgotd)$. 
Assume now we have a proper decomposition $\hgotd=\hgotdc_1\oplus \hgotdc_2$. This induces a decomposition $\sym^2(\hgotd)=\sym^2(\hgotdc_1)\oplus (\hgotdc_1\otimes\hgotdc_2)\oplus \sym^2(\hgotdc_2)$.
Since the Killing form is positive definite on $\hgotd$ we deduce that $\calX$ is not contained in $\sym^2(\hgotdc_1)\oplus(\hgotdc_1\otimes \hgotdc_2)$, otherwise the restriction of $\calX$ to $(\hgotdc_1)^\perp\times (\hgotdc_1)^\perp$ would be $0$.
\end{remark}

For a subset $I\cug S$, we denote by $W_I$ the subgroup of $W$ generated by $I$ and $P_I\supseteq B$ the parabolic subgroup corresponding to $I$. Let $\pi: G/B \raw G/P_I$ be the projection.
Then $\pi^*:H(G/P_I)\raw H(G/B)$ is injective. We can also characterize the image of $\pi^*$: it coincides with the set of $W_I$ invariants in $H(X)$, i.e.
$\pi^*(H(G/P_I))=(R/R^W_+)^{W_I}=R^{W_I}/R^W_+\cug H(X)$, and a basis is given by the set 
$$\displaystyle\{Q_v \mid v\in W\text{ has minimal length in its coset in }W/W_I\}.$$

For a simple reflection $u\in S$ let $P_u:=P_{\{u\}}$ be the minimal parabolic subgroup of $G$ containing $u$.
For any element $w\in W$ such that $\ell(wu)<\ell(w)$  we can choose a reduced expression $\underline{w}=st\ldots u$. 
The projection $\pi: G/B\raw G/P_{u}$ is a $\mathbb{P}^1$-fibration which restricts to a $\mathbb{P}^1$-fibration on $X_w$ since $\bar{BwB}\cdot P_{u}= \bar{BwB}$. The image $\pi(X_w)=X_{w}^{u}$ is the parabolic
Schubert variety of the element $w$ in $G/P_{u}$.
The intersection cohomology $IH(X^u_w)$ is a polarized Lefschetz module over $(R^u)^2\cong NS(G/P_u)$, so we can define the Lie algebra $\frg_{NS}(X^u_w):=\frg((R^u)^2,IH(X^u_w))$.

\subsection{A distinguished subalgebra of \texorpdfstring{$\frg_{NS}(w)$}{the N\'eron-Severi Lie algebra of a Schubert Variety}}\label{Subalgebra}

Let $w\in W$ and $u$ be a simple reflection such that $wu<w$. Let $\pi:G/B\raw G/P_u$ be the projection as above. 
We denote by $IC_w$ (resp. $IC_w^u$) the intersection cohomology complex for the variety $X_w$ (resp. $X_w^u$).
Then $R\pi_*(IC_w)\cong IC_{w}^{u}[1] \oplus IC_{w}^{u}[-1]$ (not canonically) by the Decomposition Theorem (the use of the Decomposition Theorem here can be avoided using an argument of Soergel 
\cite[Lemma 3.3.2]{S5}). In particular, as graded vector spaces, we have 
$IH_w\cong IH(X_w^u)\otimes H(\mathbb{P}^1)[1]$.

\begin{lemma}\label{sub1}
The Lie algebra $\frg_{NS}(w)$ contains a Lie subalgebra isomorphic to $\frg_{NS}(X_{w}^{u})$.
\end{lemma}

\begin{proof}
Let $\eta\in H^2(X_w^u)$ be the Chern class of an ample line bundle on $X_w^u$. We can apply Lemma \ref{JM0} to find a $\frsl_2$-triple $\{\pi^*\eta, h',f' \}$ inside $\frg_{NS}(w)$ such that $h'$ is of degree $0$, i.e. $h'(IH_w^k)\cug IH_w^k$ for all $k$.

Any choice of a decomposition $R\pi_*(IC_w)\cong IC_{w}^{u}[1] \oplus IC_{w}^{u}[-1]$ induces a splitting $IH_w=IH(X^u_w)[1]\oplus IH(X^u_w)[-1]$ of $R^u$-modules. 
One can easily check that weight filtration of the nilpotent element $\pi^*\eta$ is $W_k= (IH(X^u_w)[1])^{k-1}\oplus \bigoplus_{n\geq k} IH^n_w$. Therefore for  any $x\in (R^u)^2$ we have $x(W_k)\cug W_{k+2}$.

We can now apply Proposition \ref{WeirdGrading}, with $\tilde{V}=(R^u)^2$, in order to obtain
$$\frg((R^u)^2,\Grad^W(IH_w))\cong \frg(((R^u)^2)_{2,0},IH_w'),$$
 where $IH_w'$ denotes the vector space $IH_w$ with the grading determined by $h'$. In particular 
$\frg((R^u)^2,\Grad^W(IH_w))$ is a subalgebra of $\frg_{NS}(w)$.

It is easy to see that $\Grad^W(IH_w)\cong IH(X^u_w)\oplus IH(X^u_w)$ as graded vector spaces, and the isomorphism is compatible with the action of $R^u$. We conclude using Lemma \ref{double} which implies that $\frg((R^u)^2,\Grad^W(IH_w))\cong \frg_{NS}(X^u_w)$.
\end{proof}

\begin{example}
Let $G=SL_4(\bbC)$ so that  $W=\mathcal{S}_4$ is the symmetric group on $4$ elements, with simple reflections labeled $s_1,s_2,s_3$.
Let $w=s_2s_1s_3s_2$ and $u=s_2$. Let $\eta$ be an ample Chern class on $X_w^u$. Then we can draw the action of $\pi^*\eta$ on  a basis of $IH_w$ and the weight filtration as follows

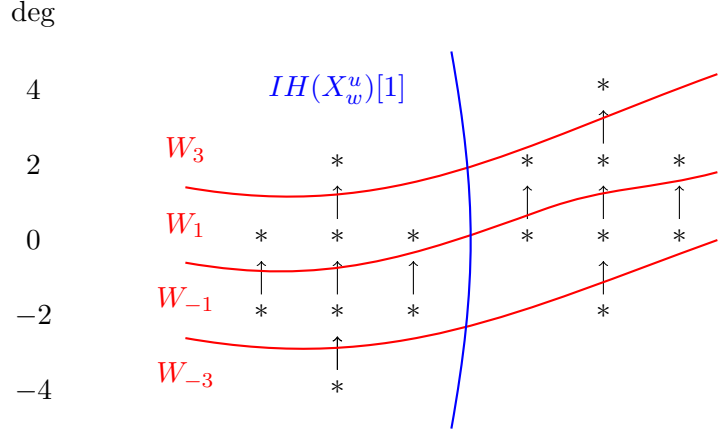
\begin{center}
\begin{tikzpicture}[auto, baseline=(current  bounding  box.center)]
      \node  (1) at (0,0) {*};
      \node (2) at (-1,1) {*};
		\node (3) at (0,1) {*};      
      \node  (4) at (1,1) {*};
      \node  (5) at (-1,2) {*};
      \node (6) at (0,2) {*};
	\node (7) at (0,3) {*};
	\node (8) at (1,2) {*};
	\node (a) at (3.5,1) {*};
	\node (b) at (2.5,2) {*};
	\node (c) at (3.5,2) {*};
	\node (d) at (4.5,2) {*};
	\node (e) at (2.5,3) {*};
	\node (f) at (3.5,3) {*};
	\node (g) at (4.5,3) {*};
	\node (h) at (3.5,4) {*};
\draw[->] (1) -- (3);
\draw[->] (3) -- (6);
\draw[->] (6) -- (7);
\draw[->] (1) -- (3);
\draw[->] (2) -- (5);\draw[->] (d) -- (g);
\draw[->] (4) -- (8);\draw[->] (b) -- (e);
\draw[->] (a) -- (c);\draw[->] (f) -- (h);
\draw[->] (c) -- (f);
\draw[red, thick] (-2,2.7) to [out=-10,in=200]  (5,4.2);
\draw[red, thick] (-2,1.7) to [out=-10,in=200]  (2.7,2.4) [out=20,in=195] to (5,2.9);
\draw[red, thick] (-2,0.7) to [out=-10,in=200]  (5,2);
\node[red] at (-2,3.2) {$W_3$};
\node[red] at (-2,2.2) {$W_1$};
\node[red] at (-2,1.2) {$W_{-1}$};
\node[red] at (-2,0.2) {$W_{-3}$};
\node at (-4,5) {$\deg$};
\node at (-4,4) {$4$};
\node at (-4,3) {$2$};
\node at (-4,2) {$0$};
\node at (-4,1) {$-2$};
\node at (-4,0) {$-4$};
\draw[blue, thick] (1.5,-0.5) to [out=80,in=-80] (1.5,4.5);
\node[blue] at (0,4) {$IH(X^u_w)[1]$};
   \end{tikzpicture}
\end{center}
\end{example}

We fix $\eta$ and $h'$  as in Lemma \ref{sub1} and let $h''=h-h'$. Then, as in Section \ref{WeightSection}, $h'$ and $h''$ define a bigrading on $IH_w$ and on $\frg_{NS}(w)$.

Notice that the only eigenvalues of $h''$ on $IH_w$ are $1$ and $-1$. It follows that $\frg_{NS}(w)$ decomposes as $\frg_{NS}(w)=\frg_{NS}(w)_{\bullet,-2}\oplus \frg_{NS}(w)_{\bullet,0}\oplus \frg_{NS}(w)_{\bullet,2}$. In particular any element $\rho$ of $R^2$ can be decomposed as $\rho=\rho_{4,-2}+\rho_{2,0}+\rho_{0,2}$.
Moreover, for $\tilde{\eta}\in (R^u)^2$ we have $\tilde{\eta}(W_k)\cug W_{k+2}$, hence $\tilde{\eta}\in \frg_{NS}(w)_{\geq 2,\bullet}$ and $\tilde{\eta}_{0,2}=\tilde{\eta}_{4,-2}+\tilde{\eta}_{2,0}$.

%We can further refine the statement of Lemma \ref{sub1}.
We can now restate and reprove \cite[Proposition 5.6]{LL} in our setting:

\begin{theorem}
The Lie algebra $\frg_{NS}(w)$ contains a Lie subalgebra isomorphic to $\frg_{NS}(X_{w}^{u})\times \mathfrak{sl}_2$. 
\end{theorem}

\begin{proof}
Take $\rho$ to be the Chern class of an ample line bundle on $X_w$. Then by the Relative Hard Lefschetz Theorem \cite[Theorem 5.4.10]{BBD} cupping with $\rho$ induces an
isomorphism of $R^u$-modules:
$$IH(X_w^u)[1]\cong{}^p H^{-1}(R\pi_*IC_w)\stackrel{\rho}{\longrightarrow}{}^p H^{1}(R\pi_*IC_w)\cong IH(X_w^u)[-1].$$

This means that the $(0,2)$-component $\rho_{0,2}\in \frg_{NS}(w)_{0,2}$ of $\rho$ (thus we have $[h',\rho_{0,2}]=0$ and $[h'',\rho_{0,2}]=2\rho_{0,2}$)  has the Lefschetz property with respect to the grading given by $h''$. 
In particular, because of Lemma \ref{commute}, we can complete it to an $\frsl_2$-triple $\{\rho_{0,2},h'',f''_\rho\}\cug \frg_{NS}(w)$. The span of $\{\rho_{0,2},h'',f''_\rho\}$ is a subalgebra of $\frg_{NS}(w)_{0,\bullet}$. In fact, since both $\rho_{0,2}$ and $h''$ commute with $h'$ so does $f''_\rho$ (see Remark \ref{RealCommute}).

%\begin{center}
%\begin{tikzpicture}
%  \matrix (m) [matrix of math nodes,row sep=3em,column sep=2.5em,minimum width=2em,text height=1.5ex, text depth=0.25ex]
%  {
%     IH^{-\ell(w)+1,-1} & IH^{-\ell(w)+3,-1} &\hspace{-10pt} \cdots \hspace{-10pt} & IH^{\ell(w)-3,-1} & IH^{\ell(w)-1,-1} \\
%     IH^{-\ell(w)+1,+1} & IH^{-\ell(w)+3,+1} &\hspace{-10pt} \cdots \hspace{-10pt} & IH^{\ell(w)-3,+1} & IH^{\ell(w)-1,+1} \\};
%  \path[->]
%    (m-1-1) edge [bend right=20] node [left] {$\rho^{0,2}$} (m-2-1)
% (m-2-1)   edge [bend right=20] node [right] {$f''_\rho$} (m-1-1)
%    (m-1-2) edge [bend right=20] node [left] {$\rho^{0,2}$} (m-2-2)
%    (m-2-2)   edge [bend right=20] node [right] {$f''_\rho$} (m-1-2)
%    (m-1-5) edge [bend right=20] node [left] {$\rho^{0,2}$} (m-2-5)
%    (m-2-5)   edge [bend right=20] node [right] {$f''_\rho$} (m-1-5)
%    (m-1-4) edge [bend right=20] node [left] {$\rho^{0,2}$} (m-2-4)
%    (m-2-4)   edge [bend right=20] node [right] {$f''_\rho$} (m-1-4)
%    (m-1-1) edge  node [above] {$\pi^*\eta$} (m-1-2) 
%     (m-1-2)  edge node [above] {$\pi^*\eta$}(m-1-3)  
%     (m-1-3)  edge (m-1-4)
%     (m-1-4) edge node [above] {$\pi^*\eta$} (m-1-5) 
%     (m-2-1) edge node [below] {$\pi^*\eta$} (m-2-2) 
%     (m-2-2)  edge node [below] {$\pi^*\eta$}(m-2-3)  
%     (m-2-3)  edge (m-2-4)
%     (m-2-4) edge node [below] {$\pi^*\eta$} (m-2-5) 
% ;
%\end{tikzpicture}
%\end{center}

Recall from Lemma \ref{sub1} that $\frg_{NS}(X^u_w)$ is isomorphic to $\frg(((R^u)^2)_{2,0},IH'_w)$, which in turn is a subalgebra of $\frg_{NS}(w)$.
It remains to show that the two subalgebras $
 \frg(((R^u)^2)_{2,0},IH'_w)$ and $\Span\{\rho_{0,2},h'',f''_\rho\}\cong \frsl_2(\bbR)$ intersect trivially and mutually commute.
Since $\rho$ commutes with $\tilde{\eta}$ for any $\tilde{\eta}\in (R^u)^2$, then also $\rho_{0,2}$ commutes with $\tilde{\eta}_{2,0}$: in fact since $\rho=\rho_{4,-2}+\rho_{2,0}+\rho_{0,2}$ and  $\tilde{\eta}=\tilde{\eta}_{4,-2}+\tilde{\eta}_{2,0}$, we have   $[\rho_{0,2},\tilde{\eta}_{2,0}]=[\rho,\tilde{\eta}]_{2,2}=0$.

Because $(R^u)^2$ and $h'$ commute with $\rho_{0,2}$, so does $\frg((R^u)^2_{2,0},IH'_w)$.
Because $\rho_{0,2}$ and $h''$ commute with $\frg((R^u)^2_{2,0},IH'_w)$, so does $f''_\rho$. We obtain a morphism of Lie algebras
$$\mathfrak{J}: \frg_{NS}(X_w^u)\times \frsl_2(\bbR)\cong \frg((R^u)^2_{2,0},IH'_w)\times\Span\{\rho_{0,2},h'',f''_\rho\} \raw \frg_{NS}(w)$$
given by the multiplication. The kernel of $\mathfrak{J}$ is $\frg_{NS}(X_w^u)\cap \frsl_2(\bbR)$ and it is contained in the center of $\frsl_2(\bbR)$, which is trivial. The thesis now follows.
\end{proof}

\subsection{Irreducibility of the subalgebra and consequences}

The goal of the first part of this section is to show the following:
\begin{prop}\label{irrmod}
$IH(X_w^u)$ is irreducible as a  $\frg_{NS}(X_w^u)$-module.
\end{prop}

We begin with a preparatory lemma:
\begin{lemma}\label{lemma34}
The cohomology $H(G/P_u)$ is generated as an algebra by the first Chern classes, i.e. by $H^2(G/P_u)$.
\end{lemma}
\begin{proof}

We can identify $H(G/P_u)$ with $R^u/(R^W_+)$. The set $\{Q_s\}_{s\in S\setminus \{u\}}$ forms a basis of $H^2(G/P_u)=NS(G/P_u)=(R^2)^u$.
It is enough to show that $\sym^2((R^2)^u)\raw H^4(G/P_u)$ is surjective, because all the generators of $H(G/P_u)$ lie in degrees $\leq 4$.

%More explicitly, if we take $\varphi\in W^\perp\cug \mathfrak{h}=(H^2(X))^*$, then
%$$0=\calX(\varphi,\varphi)=\sum_{s,t\in S}c_{st}Q_s(\varphi)Q_t(\varphi)=
%\left(\sum_{s\in S}\frac{Q_s(\varphi)}{(\alpha_s,\alpha_s)}\alpha_s,\sum_{s\in S}\frac{Q_s(\varphi)}{(\alpha_s,\alpha_s)}\alpha_s\right)$$
%thus $\varphi=0$.

The subalgebra $R^u$ is generated by $Q_s$, with $s\in S\setminus\{u\}$, and $\alpha_u^2$. Therefore $\dim (R^4)^u=\dim \text{Sym}^2((R^2)^u)+1$ and, since $H^4(G/P_u)=(R^4)^u/(\bbR\calX)$, we have $\dim H^4(G/P_u)=\dim \text{Sym}^2((R^2)^u)$. 
So it suffices to show that $\sym^2((R^2)^u)\raw H^4(G/P_u)$ is injective, or in other words that $\Ker(\sym^2((R^2)^u)\raw H^4(G/P_u))=\bbR\calX\cap \sym^2((R^2)^u)=0$, where $\calX\in (R^4)^W$ is the element defined in Lemma \ref{kilinv}.

But since the Killing form is non-degenerate and $(R^2)^u$ is a proper subspace of $R^2$, we have $\calX\not\in \text{Sym}^2((R^2)^u)$ (as explained in Remark \ref{nondeg}).
\end{proof}

\begin{proof}[Proof of Proposition \ref{irrmod}]
Since $\frg_{NS}(X_w^u)$ is semisimple, it is enough to show that $IH(X_w^u)$ is an indecomposable $\frg_{NS}(X_w^u)$-module. 
In particular it is enough to show that it is indecomposable as a $H^2(X_w^u)$-module (here regarded as an abelian Lie subalgebra of $\frg_{NS}(X^u_w)$).

The Erweiterungssatz (in the version proved by Ginzburg \cite{Gi}) states that taking the hypercohomology (as a module over the cohomology of the partial flag variety) is a fully faithful functor on  $IC$ complexes of Schubert varieties. In particular 
for any $w \in W$ we have:
$$\End_{H(G/P_u)\text{-Mod}}(IH(X_w^u))\cong \End_{D^b(G/P_u)}(IC(X_w^u)).$$
This implies, since $IC(X_w^u)$ is a simple perverse sheaf on $G/P_u$, that $IH(X_w^u)$ is an indecomposable $H(G/P_u)$-module. Now Lemma \ref{lemma34} completes the proof.
\end{proof}

\begin{remark}
Proposition \ref{irrmod}  is not true for a general parabolic flag variety. Let $G=SL_4(\bbC)$ so that  $W=\mathcal{S}_4$ is the symmetric group on $4$ elements, with simple reflections labeled $s,t,u$. 
Then $SL_4(\bbC)/P_{\{s,u\}}$ is isomorphic to $Gr(2,4)$, the Grassmannian of $2$-dimensional subspaces in $\bbC^4$.
Since $\dim H^2(Gr(2,4))=1$ we have $\frg_{NS}(Gr(2,4))\cong \frsl_2(\bbR)$, but $\dim H^4(Gr(2,4))=2$ so it cannot be irreducible as a $\frg_{NS}(Gr(2,4))$-module. In fact, $H(Gr(2,4))$ is not generated by $H^2(Gr(2,4))$.
\end{remark}
\begin{prop}\label{3.7}
If $\frg_{NS}^\bbC(w):=\frg_{NS}(w)\otimes\bbC$ is a simple complex Lie algebra, then we have $\frg_{NS}(w)\cong \mathfrak{aut}(IH_w,\phi)$.
\end{prop}
In particular this implies that the complexification $\frg_{NS}^\bbC(w)$ is isomorphic to $\mathfrak{sp}_{IH_w}(\bbC)$ if $\ell(w)$ is odd, and is isomorphic to $\mathfrak{so}_{IH_w}(\bbC)$ if $\ell(w)$ is even.
\begin{proof}
Proposition \ref{irrmod} shows that the Lie algebra $\frg_{NS}(X_{w}^{u})\times \mathfrak{sl}_2(\bbR)$ acts irreducibly on 
$IH_w\cong IH(X_w^u)\otimes H(\mathbb{P}^1)$.
This obviously remains true when one considers, after complexification, the action of $\frg_{NS}^\bbC(X_{w}^{u})\times \frsl_2(\bbC)$ on $IH(X_w,\bbC)$.

In \cite[Theorem 2.3]{Dy}, Dynkin classified all the pairs 
$\frg\cug \frg'$ ($\cug \frgl(\bbC^N)$) of complex Lie algebras such that $\frg$ acts irreducibly on $V$ and $\frg'$ is simple. From this classification we see that if  
$\frg=\widetilde{\frg}\times \frsl_2(\bbC)$ and $\frsl_2(\bbC)$ acts with highest weight $1$ then $\frg'$ is one of $\frsl_N$, $\mathfrak{so}_N$ and  $\mathfrak{sp}_N$.

We apply now this result to the pair $\frg_{NS}^\bbC(X^u_w)\times \frsl_2(\bbC)\cug \frg_{NS}^\bbC(w)$ %($\cug\frgl(IH(X_w,\bbC))$)
.
 Clearly we cannot have $\frg_{NS}^\bbC(w)\cong \frsl(IH(X_w,\bbC))$ since $\frg_{NS}(w)\cug \mathfrak{aut}(IH(X_w,\bbC), \phi)$. 
This implies $\frg_{NS}^\bbC(w)=\mathfrak{aut}(IH(X_w,\bbC), \phi)$, hence $\frg_{NS}(w)\cong \mathfrak{aut}(IH_w,\phi)$.
\end{proof}

\begin{remark}
We now discuss  which real forms of the symplectic and orthogonal groups occur as $\mathfrak{aut}(IH_w,\phi)$.
If $\ell(w)$ is odd there is, up to isomorphism, only one symplectic form on $IH_w$, hence $\mathfrak{aut}(IH_w,\phi)\cong\mathfrak{sp}_{\dim(IH_w)}(\bbR)$.

Now we assume that $\ell(w)$ is even. We want to determine the signature of the symmetric form $\phi$ on $IH_w$.

If $k> 0$ then $\phi$ is a perfect pairing between $IH_w^k$ and $IH_w^{-k}$, hence the signature of $\phi|_{IH_w^k\oplus IH_w^{-k}}$ is $(\dim IH_w^k,\dim IH_w^k)$. 
The signature of $\phi$ on $IH_w^0$ is determined by the Hodge-Riemann bilinear relations: the dimension of the positive part of $\phi|_{IH_w^0}$ is given by 
$$\sum_{i=0}^{ \lfloor l(w)/4\rfloor }\dim P^{-\ell(w)+4i}=\sum_{i=0}^{ \lfloor l(w)/4\rfloor}\left(\dim IH_w^{\ell(w)-4i}-\dim IH_w^{\ell(w)-4i+2}\right).$$

\end{remark}

\section{Tensor decomposition of intersection cohomology}

We now want to understand for which $w \in W$ the Lie algebra  $\frg_{NS}^\bbC(w)$ is not simple. The complex Lie algebra $\frg_{NS}^\bbC(w)$ acts naturally on $IH(X_w,\bbC)$. 
To simplify the notation from now on, we will consider in this section only cohomology with complex coefficients and we will denote $IH(X_w,\bbC)$ (resp. $H(X_w,\bbC)$) simply by $IH_w$ (resp. $H_w$) and $R\otimes \bbC\cong\bbC[\hgotdc]$ by $R$.

For any $w\in W$ we have $H_w\cug IH_w$ (see Remark \ref{natmap}). In particular  $H^2_w$ acts faithfully on $IH_w$ and we can regard $H^2_w$ as a subspace of $\frg_{NS}(w)$. 
We recall the following lemma from \cite[Lemma 1.2]{LL}:

\begin{lemma}
Assume there exists a non-trivial decomposition $\frg^\bbC_{NS}(w)=\frg_1\times \frg_2$ and consider $\pi_i:\frg_{NS}^\bbC(w)\raw \frg_i$ the projections. 
Then the decomposition is graded and it also induces a decomposition into graded vector spaces $IH_w=IH_w^{\bullet,0}\otimes_\bbC IH_w^{0,\bullet}$ where $IH_w^{\bullet,0}$ (resp. $IH_w^{0,\bullet}$)  is an irreducible 
$\pi_1(H^2_w)$-Lefschetz module  (resp. $\pi_2(H^2_w)$-Lefschetz module) with $\frg_1=\frg(\pi_1(H^2_w),IH_w^{\bullet,0})$  and $\frg_2=\frg(\pi_2(H^2_w),IH_w^{0,\bullet})$.
\end{lemma}

For the rest of this paper we assume that we have a splitting $\frg^\bbC_{NS}(w)=\frg_1\times \frg_2$ and we denote by $\pi_1:\frg_{NS}^\bbC(w)\raw \frg_1$  and $\pi_2:\frg_{NS}^\bbC(w)\raw \frg_2$ the projections.
Let $IH_w=IH_w^{\bullet,0}\otimes_\bbC IH_w^{0,\bullet}$ be the induced decomposition.

There exist integers $a,b\geq 0$ such that $IH_w^{\bullet,0}$ (resp. $IH_w^{0,\bullet}$) are not trivial only in degrees between $-a$ and $a$ (resp. between $-b$ and $b$) with $a,b\geq 0$ and $a+b=\ell(w)$. 
In particular $IH_w^{-a,0}$ and $IH_w^{0,-b}$ are one dimensional.
We define a bigrading on $IH_w$ by $IH_w^{i,j}:=IH^{i,0}_w\otimes IH_w^{0,j}$.

%The restriction of the projection $p_1:IH_w^{-\ell(w)+2}\raw IH_w^{-a+2,-b}$ to $H_w^2$ is compatible with the restriction of $\pi_1$ to $H_w^2$ (and similarly for $\pi_2$). 
%This means that for $\lambda\in H^2_w$ we have $p_1(\lambda\cdot 1_w)=p_1(\lambda(1_w))=\pi_1(\lambda)\cdot 1_w$.
%Here $1_w$ denotes a generator of the one dimensional vector space $IH_w^{-\ell(w)}=IH_w^{-a,-b}$.
%In particular, for $\lambda \in H^2_w$, we have
%$\lambda\in \frg_1$ (resp. $\lambda\in \frg_2$)  if and only if $\lambda\cdot 1_w\in IH_w^{\bullet,-b}$
%(resp. $\lambda\cdot 1_w\in IH_w^{-a,\bullet}$). 

\subsection{Splitting of \texorpdfstring{$H^2_w$}{H\textasciicircum 2\_w}}

We can assume from now on $H^2_w=H^2(G/B)$. In fact, we can replace $G$ by its Levi subgroup corresponding to the smallest parabolic subgroup of $G$ containing $w$. 
This does not change the Schubert variety $X_w$, the cohomology $H_w$ and the Lie algebra $\frg_{NS}(w)$. In particular we have $R=\sym(H^2_w)$.

In general $H_w\neq IH_w$,  so it is not clear a priori that a tensor decomposition for $IH_w$ descends to one for $H_w$. 
Still, this holds in our setting:

\begin{prop}\label{split}
Assume we have a decomposition $\frg_{NS}^\bbC(w)=\frg_1\times \frg_2$.  Then $H^2_w=\pi_1(H^2_w)\oplus \pi_2(H^2_w)$.
%(this is clear if the Schubert variety $X_w$ since in this case $IH^2(X_w)=H^2(X_w)=NS(X_w)$). This will be done by showing that $\dim NS(X)\geq \dim \pi_1(NS(X)) +\dim \pi_2(NS(X))$. 
\end{prop}
\begin{proof}
It is enough to show that $\dim H^2_w\geq \dim \pi_1(H^2_w) +\dim \pi_2(H^2_w)$. 
We define 
$$T:=\text{Sym}(\pi_1(H^2_w))\otimes \text{Sym}(\pi_2(H^2_w))\cong\text{Sym}(\pi_1(H^2_w)\oplus \pi_2(H^2_w)).$$ 
We can define a $T$-module structure on $IH_{w}$ via $(x\otimes y)(a)=x(a)\otimes y(a)$ for any $x\in \pi_1(H_w^2)$, $y\in \pi_2(H_w^2)$ and $a\in IH_w$.

We have a bigrading  $T^{p,q}:=\text{Sym}^p(\pi_1(H^2_w))\otimes \text{Sym}^q(\pi_2(H^2_w))$ on $T$ compatible with the bigrading of
$IH_w$, i.e. $T^{p,q}(IH_w^{i,j})\cug IH_w^{p+i,q+j}$.

The subspace $T^{2,0}\cong\pi_1(H_w^2)\cug \frg_1$ acts faithfully on $IH_w^{\bullet,0}$, while $T^{0,2}\cong\pi_2(H_w^2)\cug \frg_2$ acts faithfully on $IH_w^{0,\bullet}$. 
Hence $T^{2,2}\cug \frg_1\otimes \frg_2\cug \frgl (IH_w^{\bullet,0})\otimes  \frgl (IH_w^{0,\bullet})=\frgl(IH_w)$ acts faithfully on $IH_w$, i.e. if $t\in T^{2,2}$ acts as $0$ on $IH_w$, then $t=0$. 

Let $\Psi:R\hookrightarrow T$ the inclusion induced by $\Psi(x)= \pi_1(x)+\pi_2(x)$ for any $x\in R^2$. We observe that the $T$-module structure on $IH_w$ extends the $R$-module structure.

We can decompose $Q_s=L_s+R_s$ where $L_s=\pi_1(Q_s)\in \frg_1$ and $R_s=\pi_2(Q_s)\in \frg_2$ for all $s\in S$.
Now we consider the element $\calX\in (R^4)^W$ defined in Lemma \ref{kilinv}. The $R$-module structure on $IH_w$ factorizes through $H(X,\bbC)=R/(R_+^W)$, therefore  $\Psi(\calX)\in T$ acts as $0$ on $IH_w$. 
In particular also the component $\Psi(\calX)^{2,2}\in T^{2,2}$ acts as $0$ on $IH_w$. Since the action is faithful on $T^{2,2}$ we obtain
$\Psi(\calX)^{2,2}=\sum_{s,t\in S}c_{st}(L_s\otimes R_t+L_t\otimes R_s)=0\in T^{2,2}$. Since $c_{st}$ is symmetric we can rewrite it as follows:

$$\sum_{s,t\in S}L_s\otimes c_{st}R_t=0 \in \pi_1(H^2_w)\otimes \pi_2(H^2_w)\cug \frg_1\otimes \frg_2.$$

Let $S_L\cug S$ be such that $\{L_s\}_{s\in S_L}$ is a basis of $\pi_1(H^2_w)$. Writing $L_u=\sum_{s\in S_L}x_{su}L_s$ with $x_{su}\in \bbR$ for $u\in S\setminus S_L$ we get 

$$\sum_{\substack{s \in S_L\\ t\in S}}L_s\otimes \left(c_{st} +\sum_{u\in S\setminus S_L} x_{su}c_{ut}\right)R_t=0\implies \sum_{t\in S} \left(c_{st} +\sum_{u\in S\setminus S_L} x_{su}c_{ut}\right)R_t=0$$ for any $s\in S_L$.
Since $(c_{st})_{s,t\in S}$ is a non-degenerate matrix, it follows that we have $\#(S_L)$ linearly independent equations vanishing on $(R_s)_{s\in S}$, 
hence $\dim \pi_2(H^2_w)\leq \dim H^2_w-\#(S_L)=\dim H^2_w- \dim \pi_1(H^2_w)$.
\end{proof}

It also follows that $\Psi: R\raw T$ is an isomorphism, so we have a bigrading on $R$ compatible with the one on $IH_w$.

Hence $H_w$ is also bigraded as a subspace of $IH_w$, since it is the image of 
the map of bigraded vector spaces $R\raw IH_w$ induced by $x\mapsto x(1_w)$, where $1_w$ is a generator of the one dimensional space $IH_w^{-\ell(w)}$.

In the next sections we provide a sufficient condition 
for the Lie algebra $\frg_{NS}(w)$ to be maximal. However there is a case where the proof is considerably easier and we provide it here for convenience and to motivate the reader.

Recall that for any $w\in W$, the set $\{Q_{st}\}_{st\leq w}$ is a basis of $H^4_w$. In particular, if $st\leq w$ for any $s,t\in S$, we have $H^4_w\cong H^4(X)$. In this case from Lemma \ref{kilinv} we have also $\Ker(R^4\raw H^4_w)=(R^W_+)^4=\mathbb{R}\mathcal{X}$.

\begin{cor}\label{cor2}
Assume that the root system of $G$ is irreducible and suppose that whenever $s_i,s_j\leq w$ then $s_is_j\leq w$. Then $\frg_{NS}(w)\cong \mathfrak{aut}(IH_w,\phi)$.
\end{cor}
\begin{proof}
%Since $H^2_w=H^2(G/B)$, the hypotheses imply that also $H^4_w=H^4(X)$. Since the Dynkin diagram is irreducible, we have $(R^W_+)^4\cong \bbC\calX$. Then the map $\sym^2(H^2_w)\raw H^4_w$ is surjective and the kernel is generated by $\calX$. 

We assume for contradiction that we have a non-trivial decomposition $\frg_{NS}^\bbC(w)=\frg_1\times \frg_2$. From Proposition \ref{split}  we know that $H^4_w$ splits as $H^{4,0}_w\oplus H_w^{2,2}\oplus H^{0,4}_w$.
This implies that also $K:= \Ker(R^4\raw H^4_w)$ splits as $K=K^{4,0}\oplus K^{2,2}\oplus K^{0,4}$ where $K^{i,j}=\Ker(R^{i,j}\raw H^{i,j}_w)$. 
But $K$ is one dimensional and generated by $\calX$, thus $\calX$ belongs to either $R^{4,0}$, $R^{2,2}$ or $R^{0,4}$, which is impossible since $\calX$ is non-degenerate (see Remark \ref{nondeg}). Hence the Lie algebra $\frg_{NS}^\bbC(w)$ must be simple. We can now apply proposition \ref{3.7} to deduce $\frg_{NS}(w)\cong \mathfrak{aut}(IH_w,\phi)$.
\end{proof}

\subsection{A directed graph associated to an element}
Let $w\in W$. We construct an oriented graph $\mathcal{I}_w$ as follows: the vertices are indexed by the set of simple reflections $S$ and we put an arrow $s\raw t$ if $ts\leq w$ and $ts\neq st$ 
(i.e. if $ts\leq w$ and $s$ and $t$ are connected in the Dynkin diagram). 

Recall that we assumed, by shrinking to a Levi subgroup, that $s\leq w$ for any $s\in S$. It follows that for any pair $s,t\in S$ we have either $st\leq w$, $ts\leq w$ or both. Hence the graph $\mathcal{I}_w$ is just the Dynkin diagram where each edge $s-t$ is replaced by the arrow $s\leftarrow t$, by the arrow $s\raw t$, or by both $s\rightleftarrows t$.
In particular, if the Dynkin diagram is connected, then also $\mathcal{I}_w$ is connected. In this case we call $w$ \textit{connected}. 

\begin{remark}\label{loops}
Since the Dynkin diagram has no loops, then also $\mathcal{I}_w$ has no non-oriented loops (we only consider loops in which for any pair $s,t\in S$ at most one of the arrows $s\raw t$ and $t\raw s$ occurs). 
\end{remark}

We call a subset $C\cug S$ \emph{closed} if any arrow in $\mathcal{I}_w$ starting in $C$ ends in $C$.
Union and intersection of closed subsets are still closed.
We call a closed singleton in $S$ a \textit{sink}.

\begin{example}
Let $W$ be the Coxeter group of $D_5$. We label the simple reflections as follows:
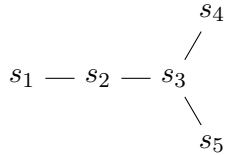
\[ \begin{tikzpicture}[auto, baseline=(current  bounding  box.center)]
      \node  (u) at (0,0) {$s_3$};
      \node (t) at (180:1) {$s_2$};
      \node  (v) at (60:1) {$s_4$};
      \node  (w) at (-60:1) {$s_5$};
	\node (s) at (180:2) {$s_1$};

      \draw (t) to (s);
      \draw (t) to (u);
      \draw (u) to (v);
 \draw (u) to (w);
   \end{tikzpicture} \]

Consider the element $w=s_1s_2s_4s_3s_5s_2s_1$. Then the diagram $\mathcal{I}_w$ associated to $w$ is:
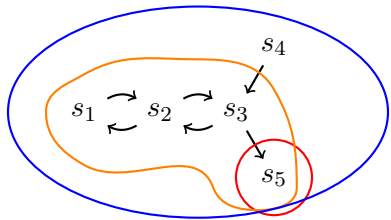
\[ \begin{tikzpicture}[auto, baseline=(current  bounding  box.center)]
      \node  (u) at (0,0) {$s_3$};
      \node (t) at (180:1) {$s_2$};
      \node  (v) at (60:1) {$s_4$};
      \node  (w) at (-60:1) {$s_5$};
	\node (s) at (180:2) {$s_1$};

      \draw [->,out=-150, in=-30, thick] (t) to (s);
 \draw [->,out=30, in=150, thick] (s) to (t);
   \draw [->,out=-150, in=-30, thick ] (u) to (t);
 \draw [->,out=30, in=150, thick] (t) to (u);
      
      \draw [->, thick] (v) to (u);
 \draw [->, thick] (u) to (w);

\draw [red, thick] (-60:1) circle [radius=0.5];
\draw [orange, thick] (0.8,-1) to [out=-90, in=0] (0.3,-1.3) to [out=180, in=-70] (-0.3,-1) to [out=110, in=-30] (-2,-0.7) to [out=150, in=-90] (-2.5, 0) to 
[out=90, in=210] (-2,0.5) to [out=30, in=180] (-0.5,0.7) to [out=0, in=120] (0.5,0.3) to [out=-60, in=90] (0.8,-1);
\draw [blue, thick] (-0.5,0) ellipse (2.5 and 1.4);
   \end{tikzpicture} \]
Here the coloured lines describe all  the non-empty closed subsets of $\mathcal{I}_w$.
\end{example}

As we show in the following sections, the graph $\mathcal{I}_w$ determines $H^4_w$, and we can make use of it to provide obstructions for the algebra $\frg_{NS}(w)$ to not admit a decomposition, hence find sufficient conditions for the algebra $\frg_{NS}(w)$ to be simple.
More specifically, we prove in Theorem \ref{Theorem} that, if $\mathcal{I}_w$ is connected and has no sinks, then $\frg_{NS}(w)$ is maximal.

\subsection{Reduction to the connected case}

If $w$ is not connected,  we can write $w=w_1w_2$, with $\ell(w)=\ell(w_1)+
\ell(w_2)$ such that $(\alpha_{s_1},\alpha_{s_2})=0$ for any $s_1\leq w_1, s_2\leq w_2$.

\begin{prop}\label{nonconnected}
If $w=w_1w_2$ as above, then we have decompositions $IH_w\cong IH_{w_1}\otimes_\bbC IH_{w_2}$ and $\frg_{NS}(w)\cong \frg_{NS}(w_1)\times \frg_{NS}(w_2)$.
\end{prop}
\begin{proof}
In this case $X_w\cong X_{w_1}\times X_{w_2}$, so $IH_w=IH_{w_1}\otimes IH_{w_2}$. Moreover $H_w^2=H_{w_1}^2\oplus H_{w_2}^2$ where $H_{w_1}$ acts on the factor $IH_{w_1}$ while $H_{w_2}$ acts on $IH_{w_2}$.
Since the Lie algebra $\frg_{NS}(w_1)\times \frg_{NS}(w_2)$ is semisimple and both $h$ and $H^2_w$ are contained in $\frg_{NS}(w_1)\times \frg_{NS}(w_2)$, from Lemma \ref{commute} we have $\frg_{NS}(w)= \frg_{NS}(w_1)\times \frg_{NS}(w_2)$.
\end{proof}

\subsection{The connected case}

In view of Proposition \ref{nonconnected} we can restrict ourselves to the case of a connected $w$.

\begin{lemma}
Let $w$ be connected and let $K=\Ker(\sym^2(H^2_w)\raw H^4_w)$. Then the elements $\calX_C:=\sum_{s,t\in C}c_{st}Q_sQ_t$, with $C$ closed,  generate $K$.
\end{lemma}
\begin{proof}
We know that $\dim K=\#\{(s,t)\in S^2\mid st\not\leq w\}+1$ because $\sym^2(H^2_w)\raw H^4_w$ is surjective. Since $w$ is connected, if $st\not\leq w$ then $s$ and $t$ are connected by an edge in the Dynkin diagram and $ts\leq w$.

Let $(a,b)$ be any pair of elements of $S$ such that $ab\not\leq w$, i.e. such that there is no arrow $b\raw a$. We can define a proper closed subset $C_{ab}$ by taking the connected component of $b$ in $\mathcal{I}_w$ after erasing the arrow $a\raw b$. From Remark \ref{loops} it follows that $a\not\in C_{ab}$.
It is easy to see that $\mathcal{X}_{C_{ab}}$ together with $\calX=\calX_S$ are linearly independent in $\sym^2(H^2_w)$:
in fact when we write them in the basis $\{Q_sQ_t\}_{s,t\in S}$ we have $\calX_{C_{ab}}\in c_{bb}Q_b^2+\mathcal{R}_{ab}$, where $\mathcal{R}_{ab}=\Span \langle Q_sQ_t\mid (s,t)\neq (a,a),(b,b)\rangle$, 
while all the other $\calX_{C_{a'b'}}$ are either in $\mathcal{R}_{ab}$ or in $c_{aa}Q_a^2+c_{bb}Q_b^2+\mathcal{R}_{ab}$.

By the formula for the dimension of $K$ given above, it remains to show that all the $\calX_C$, for $C$ closed, lie in $K$. 
Let $\bar{y}$ denote the projection to $H^4(G/B)$ of an element $y\in \sym^2(H^2_w)$.
Let $C$ be a closed subset and let $E:=\{a(i)\stackrel{i}{\raw} b(i)\mid a(i)\not \in C\text{ and }b(i)\in C\}$ be the set of arrows starting outside $C$ and ending in $C$.
Applying Lemma \ref{conti}, on one hand we obtain:
\begin{equation}\label{eq1}
 \bar{\calX_C}=\sum_{s,t\in C}c_{st}\bar{Q_sQ_t}\in \Span\langle Q_{st}\mid s,t\in C\rangle\oplus \Span\langle Q_{a(i)b(i)}\mid i \in E\rangle\cug H^4(G/B).
 \end{equation}

On the other hand we have
$$\calX-\calX_C=\sum_{s,t\not \in C}c_{st}Q_sQ_t+\sum_{i\in E}2c_{a(i)b(i)}Q_{a(i)}Q_{b(i)}\in \sym^2(H^2_w).$$

Since $\bar{\calX}=0$ in $H^4(G/B)$, projecting from $R^4$ to $H^4(G/B)$ we obtain
\begin{equation}\label{eq2}
\bar{\calX_C}\in \Span \langle Q_{st} \mid s,t \not \in C\rangle \oplus \Span\langle Q_{a(i)b(i)}\mid i \in E\rangle\oplus\Span \langle Q_{b(i)a(i)}\mid i \in E\rangle.
\end{equation}

Then \eqref{eq1} together with \eqref{eq2} implies that the projection $\bar{\calX_C}$ of $\calX_C$ to $H^4(G/B)$ lies in $\Span\langle Q_{a(i)b(i)}\mid i \in E\rangle$. 
But, for any $i \in E$, $Q_{a(i)b(i)}$ projects to $0$ in $H^4_w$ since $a(i)b(i)\not\leq w$, whence $\calX_C\in K$.
\end{proof}

For a closed $C$ let $NS(C):=\text{span}\langle Q_s\mid s\in C\rangle\cug H^2_w$.
The proof of  Proposition \ref{split} applies also to $NS(C)$ if we replace $\calX$ by $\calX_C=\sum_{s,t\in C}c_{st}Q_sQ_t$. This means that whenever we have a decomposition $\frg_{NS}^\bbC(w)=\frg_1\times \frg_2$, then $NS(C)$ splits compatibly.

\begin{lemma}
Let $K_C:=K\cap \sym^2(NS(C))$. Then $K_C$ is generated by $\calX_D$, with $D$ closed and $D\cug C$.
\end{lemma}
\begin{proof}
Let $\sum_{i}a_i\calX_{D_i}\in K\cap \sym^2(NS(C))$ with $D_i$ closed and $a_i\in \bbC$. Then it is easy to see that $\sum_{i}a_i\calX_{D_i}=\sum_{i}a_i\calX_{D_i\cap C}\in \sym^2(NS(C))$.
\end{proof}

For any $s\in S$, let $L_s=\pi_1(Q_s)\in \frg_1$ and $R_s=\pi_2(Q_s)\in \frg_2$.

\begin{lemma}\label{casobase}
Let $C$ be a connected and closed subset of $S$. Assume that there exists a non-empty closed subset $D\cug C$ such that $NS(D)=\pi_1(NS(C))$.
Then if $D$ does not contain any sink we have $D=C$.
\end{lemma}
\begin{proof}

Let $U=C\setminus D$ and $E:=\{a(i)\stackrel{i}{\raw} b(i)\mid a(i) \in U\text{ and }b(i)\in D\}$ be the set of arrows starting in $U$ and ending in $D$. 
The set $\{L_s\}_{s\in D}=\{Q_s\}_{s\in D}$ is a basis of $NS(D)=\pi_1(NS(C))$, therefore  the set $\{R_u\}_{u\in U}$ is a basis of $\pi_2(NS(C))$. 
We assume for contradiction that $U\neq \emptyset$. By writing the $(2,2)$-component of $\calX_C-\calX_{D}$ we have

$$\sum_{u\in U}\left(\sum_{s\in C} c_{su} L_s\right)\otimes R_u=0\in \frg_1\otimes \frg_2 $$
from which we get $\sum_{s\in C}c_{su} L_s=0$ for any $u\in U$.

Let $\tilde{U}$ be a connected component of $U$ and let $\tilde{E}=\{a(i)\stackrel{i}{\raw} b(i)\mid a(i) \in \tilde{U}\text{ and }b(i)\in D\}\cug E$. Since $C$ is connected we have $\tilde{E}\neq \emptyset$. Since $\tilde{U}$ is connected and there are no loops in the Dynkin diagram, we have $b(i)\neq b(j)$ for any $i\neq j\in \tilde{E}$, and moreover there are no arrows between $b(i)$ and $b(j)$.
Then for any $u \in \tilde{U}$ we have 
$$0=\sum_{s\in C}c_{su} L_s=\sum_{s\in \tilde{U}}c_{su}L_s+\sum_{i\in \tilde{E}}c_{b(i)u}L_{b(i)}.$$

Since the set $\{L_{b(i)}\}_{i\in \tilde{E}}$ is linearly independent, this can be thought as a non-degenerate system of linear equations in  $L_s$, with $s\in \tilde{U}$  and it has a unique solution
$$L_s=\sum_{i \in \tilde{E}}y(s,i)L_{b(i)}=\sum_{i \in \tilde{E}}y(s,i)Q_{b(i)}\text{ with }y(s,i)\in\bbR.$$
In particular

\begin{equation}\label{yfw}
\sum_{s\in \tilde{U}}y(s,i)c_{su}=\begin{cases}0 & \text{if } u\neq a(i) \\ -c_{a(i)b(i)} &\text{if }u=a(i)\end{cases}\qquad \forall u\in \tilde{U}, \forall i \in \tilde{E}.
\end{equation}

\begin{claim}
 We have $y(s,i)>0$ for any $s\in \tilde{U}$ and any $i\in \tilde{E}$.
\end{claim}
\begin{proof}[Proof of the claim]
From Equation \eqref{yfw} it is easy to see that 
$$\left(\sum_{s\in \tilde{U}}\frac{y(s,i)}{(\alpha_s,\alpha_s)}\alpha_s,\alpha_u\right)=-\delta_{a(i),u}c_{a(i)b(i)}(\alpha_u,\alpha_u)\qquad \forall u\in \tilde{U}, \forall i \in \tilde{E}.$$
Hence $\sum_{s\in \tilde{U}}\frac{y(s,i)}{(\alpha_s,\alpha_s)}\alpha_s$ is (up to a positive scalar) equal to the fundamental weight of $a(i)$ in the root system generated by the simple roots in $\tilde{U}$.
Now the claim follows from the fact that in any irreducible root system all the fundamental weights have only positive  coefficients when expressed in the basis of simple roots.
\end{proof}

For any $s\in \tilde{U}$ we have  $R_s=Q_s-\sum_{i\in \tilde{E}}y(s,i)Q_{b(i)}\in \frg_2$. Now consider the element
$$R^{0,4}\ni\hspace{-0.3pt}\sum_{s,t\in \tilde{U}}c_{st}R_sR_t=\sum_{s,t\in \tilde{U}}c_{st}\left(Q_s-\sum_{i\in \tilde{E}}y(s,i)Q_{b(i)}\right)\left(Q_t- \sum_{i\in \tilde{E}}y(t,i)Q_{b(i)}\right)=$$
\begin{equation*}
\begin{split}=\left(\sum_{s,t\in \tilde{U}}c_{st}Q_sQ_t\right)-2\sum_{i\in \tilde{E}}\left(\sum_{s,t\in \tilde{U}}y(s,i)c_{st}Q_t\right)Q_{b(i)}+\\\sum_{i,j\in \tilde{E}}\left(\sum_{s,t\in \tilde{U}} y(s,i)y(t,j)c_{st}\right)Q_{b(i)}Q_{b(j)}=\end{split}
\end{equation*}
$$=\left(\sum_{s,t\in \tilde{U}}c_{st}Q_sQ_t\right)+2\sum_{i\in \tilde{E}}c_{a(i)b(i)}Q_{a(i)}Q_{b(i)}-\sum_{i,j\in \tilde{E}}y(a(j),i)c_{a(j)b(j)}Q_{b(i)}Q_{b(j)}=$$
$$=\calX_{D\cup \tilde{U}}-\calX_{D}+\Theta\qquad\text{ with }\Theta:=-\sum_{i,j\in \tilde{E}}y(a(j),i)c_{a(j)b(j)}Q_{b(i)}Q_{b(j)}.$$

Let $p:R^4\raw H^4_w$ denote the projection.
The previous equation implies that 
$$p\left(\sum_{s,t\in \tilde{U}}c_{st}R_sR_t\right)=p(\Theta).$$
But $p(\sum_{s,t\in \tilde{U}}c_{st}R_sR_t)\in H^{0,4}_w$ while $p(\Theta)\in H_w^{4,0}$, because $b(i)\in D$ and $Q_{b(i)}\in H_w^{2,0}$ for any $i\in \tilde{E}$. It follows that $p(\Theta)\in H^{4,0}\cap H^{0,4}=\{0\}$.

We can write $\Theta=\Theta_1+\Theta_2$ with 
$$\Theta_1=\sum_{\stackrel{i,j\in\tilde{E}}{i\neq j}}y(a(j),i)c_{a(j)b(j)}Q_{b(i)}Q_{b(j)}\quad\text{ and }\quad\Theta_2=\sum_{i\in \tilde{E}}y(a(i),i)c_{a(i)b(i)}Q_{b(i)}^2.$$
Since there are no edges between $b(i)$ and $b(j)$, we have that  $p(Q_{b(i)}Q_{b(j)})=Q_{b(i)b(j)}$ for any $i,j\in\tilde{E}$ such that $i\neq j$. Thus, by Lemma \ref{conti}, we have
$$p(\Theta_1)=\sum_{\stackrel{i,j\in\tilde{E}}{i\neq j}}y(a(j),i)c_{a(j)b(j)}Q_{b(i)b(j)}$$
$$p(\Theta_2)=-2\sum_{i\in \tilde{E}}y(a(i),i)c_{a(i)b(i)}\left(\sum_{j\in E_i}\frac{(\alpha_{b(i)},\alpha_{\beta_i(j)})}{(\alpha_{\beta_i(j)},\alpha_{\beta_i(j)})}Q_{\beta_i(j)b(i)}\right)$$
where $E_i=\{b(i)\stackrel{j}{\raw}\beta_i(j)\}$ is the set of arrows in $\mathcal{I}_w$ starting in $b(i)$. It is easy to see that all the terms in $p(\Theta_1)$ and $p(\Theta_2)$ are linearly independent, whence $p(\Theta_1)+p(\Theta_2)=0$ if and only if all their terms vanish.
Recall that $y(a(i),i)c_{a(i)b(i)}<0$ for all $i\in \tilde{E}$. Hence $p(\Theta_1)+p(\Theta_2)=0$ forces $E_i=\emptyset$ for any $i\in \tilde{E}$. But this is a contradiction because there are no sinks in $D$, whence $U=\emptyset$ and $C=D$.
\end{proof}

\begin{lemma}\label{corbase}
Let $C$ be a closed and connected subset of $S$. Assume that there are no sinks in $C$. Then $NS(C)\cug\frg_1$ or $NS(C)\cug\frg_2$.
%assume that there exists a maximal proper subset $D$ such that $NS(D)\cu\frg_1$. 
%Then if $D$ is not a singleton $\{d\}$ we have $NS(C)\cug \frg_1$.
\end{lemma} 
\begin{proof}
We work by induction on the number of vertices in $C$. There is nothing to prove if $C=\emptyset$.

Let $D\cug C$ be a maximal proper closed subset. 
The kernel $K_C:=K\cap \sym^2(NS(C))$ is generated by $\calX_C$ and $\calX_{D'}$ with $D'\cug D$. In fact if $\tilde{D}\cug C$ is a proper closed subset and $\widetilde{D}\not\cug D$, then by 
maximality $\tilde{D}\cup D=C$ and $\displaystyle \calX_{\tilde{D}}=\calX_C-\calX_{D}+\calX_{D \cap \tilde{D}}$. In particular we have $\dim K_C=\dim K_D+1$.

By induction on the number of vertices we can subdivide $D$ into two subsets $D_L$ and $D_R$, each consisting of the union of connected components of $D$, such that $NS(D_L)\cug \frg_1$ and $NS(D_R)\cug \frg_2$.

Since $NS(C)$ splits, then $K_C$ also  splits as $K_C^{4,0}\oplus K_C^{2,2} \oplus K_C^{0,4}$ where $K_C^{i,j}=K_C\cap R^{i,j}$. 
However $K_C^{2,2}\cug K^{2,2}=0$ since $R^{2,0}\otimes R^{0,2}$ is mapped isomorphically to $H^{2,2}_w$. 
Using $\dim K_C=\dim K_D +1$ we get $K_C\cap R^{4,0}=K_{D}\cap R^{4,0}$ or $K_C\cap R^{0,4}=K_{D}\cap R^{0,4}$. We can assume $K_C\cap R^{4,0}=K_{D}\cap R^{4,0}=K_{D_L}$.

It follows that $\calX_C\in \sym^2\left(NS(D_L)\oplus \pi_2(NS(C))\right)$. Again, since $\calX_C$ is non-degenerate on $NS(C)$, we get
$NS(D_L)=\pi_1(NS(C))$.
Now we can apply Lemma \ref{casobase}: if $D_L\neq \emptyset$, then $D_L=C$, otherwise $\pi_1(NS(C))=0$ and $NS(C)\cug \frg_2$.
\end{proof}

\begin{theorem}\label{Theorem}
For $w\in W$, if the graph $\mathcal{I}_w$ is connected and without sinks, then $\frg_{NS}(w)=\mathfrak{aut}(IH_w,\phi)$.
\end{theorem}
\begin{proof}
Applying Lemma \ref{corbase} to $C=S$ we see that any decomposition of  $\frg_{NS}^\bbC(w)$ must be trivial, hence by Proposition \ref{3.7} we get $\frg_{NS}(w)=\mathfrak{aut}(IH_w,\phi)$.
\end{proof}

\begin{example}
It is in general false that $\frg_{NS}(w)$ is simple for any connected $w$. 

Let $W$ be the Weyl group of type $A_3$ (i.e. $W=S_4$) where $S=\{s,t,u\}$. 
We consider the element $usts\in W$ whose graph $\mathcal{I}_{usts}$ is 

\[ \begin{tikzpicture}[auto, baseline=(current  bounding  box.center)]
      \node  (u) at (0,0) {$u$};
      \node (t) at (180:2) {$t$};
      	\node (s) at (180:4) {$s$};

      \draw [->,out=-160, in=-20, thick] (t) to (s);
 \draw [->,out=20, in=160, thick] (s) to (t);
 \draw [->, thick] (t) to (u);
         \end{tikzpicture} \]

The closed subsets in $\mathcal{I}_{usts}$ are $S$, $\{u\}$ and $\emptyset$. Then $\frg_{NS}(usts)\cong \frg_{NS}(u)\times \frg_{NS}(sts)\cong \mathfrak{sp}_2(\bbR)\times  \mathfrak{sp}_{6}(\bbR)\cong \frsl_2(\bbR)\times \mathfrak{sp}_{6}(\bbR)$. 
The splitting induced on $H^2_w$ is $$H_w^2=\pi_1(H_w^2)\oplus \pi_2(H_w^2)=\bbC Q_{u}\oplus \left(\bbC(Q_{t}-\frac23Q_{u})+\bbC(Q_{s}-\frac13Q_{u})\right).$$

We have a similar behaviour more generally: for any $w\in S_{n+1}$, with $S=\{s_1,\ldots, s_n\}$, such that $w=s_1w'$ where $w'$ is the longest element in $W_{\{s_2,\ldots,s_n\}}$ the Lie algebra $\frg_{NS}(w)$ is isomorphic to $\frsl_2(\bbR)\times \frg_{NS}(w')$.
\end{example}

\begin{example}
The following example demonstrates that having no sinks in $\mathcal{I}_w$ is not a necessary condition for the algebra $\frg_{NS}(w)$ to be simple.

Let $W$ be the Weyl group of type $B_3$, where we label the simple reflections as follows:
$$\begin{tikzpicture}[auto, baseline=(current  bounding  box.center)]
      \node  (s) at (0,0) {$s$};
      \node (t) at (1,0) {$t$};
	\node (u) at (2,0) {$u$};

\draw (0.15, 0) to (0.85, 0);
      \draw (1.15, 0.05) to (1.85, 0.05);
\draw (1.15, -0.05) to (1.85, -0.05);

\end{tikzpicture}$$
Then for $w_1=usts$ we get again  $\frg_{NS}(w_1)\cong \frg_{NS}(u)\times \frg_{NS}(sts)\cong \frsl_2(\bbR)\times \mathfrak{sp}_{6}(\bbR)$, but for $w_2=stut$ the Lie algebra $\frg_{NS}(w_2)$ is simple (hence it is isomorphic to $\mathfrak{so}_{6,6}(\bbR)$).
Notice that the graphs $\mathcal{I}_{w_1}$ and $\mathcal{I}_{w_2}$ are isomorphic.
 
\end{example}

\begin{remark}

The results given in this section work in the same way, replacing the cohomology of $X$ with the coinvariant ring $R/R_+^W$
% $\Lambda\otimes_R\mathbb{R}$, where $\Lambda$ s the dual of the nil-Hecke ring (see \cite[Proposition 2.7]{KK2}) 
and the intersection cohomology of Schubert variety by indecomposable Soergel modules whenever is needed, for a finite Coxeter group $W$: 
if there are not sinks in the diagram of $w\in W$ then $\frg_{NS}(w)$ is maximal, i.e. it coincides with $\mathfrak{aut}(\bar{B_w},\phi)$. 
To complete the proof one needs to generalize Proposition \ref{irrmod} in this setting. A possible way to achieve this is to extend the results in \cite{EW1} to the setting of singular Soergel bimodules \cite{W4}.

For a general Coxeter group $W$ our methods do not apply directly. In fact in general a reflection faithful representation of $W$ is not irreducible, thus Lemma \ref{kilinv} does not hold and the kernel of the map $R\raw \bar{B}_w$ seems harder to compute.
\end{remark}

\appendix

\section{Appendix: Extension of Corollary \ref{corPD} to a general Coxeter group}\label{appendix}
The goal of this Appendix is to extend Corollary \ref{corPD} to a general Coxeter group $W$. In the general case we cannot use the geometry of the Schubert varieties to construct a graded $R$-submodule $H_w$ of $\bar{B_w}$ such that $\dim (H_w)^k=\#\{v\in W \mid v\leq w\text{ and }2\ell(v)=k+\ell(w)\}$. In this section we construct an algebraic replacement of such a module.

\subsection{A basis of the Bott-Samelson bimodule}

We use the diagrammatic notation for morphisms between Soergel bimodules from \cite{EW2}. 

For any word $\undw=s_1\ldots s_\ell$ we have the Bott-Samelson bimodule $BS(\undw)=R\otimes_{R^{s_1}}R\otimes_{R^{s_2}} R\ldots R\otimes_{R^{s_\ell}} R$ and for any $w\in W$ let $B_w$ denote the corresponding indecomposable Soergel bimodule. We have $BS(\undw)\otimes_R \mathbb{R}=\bar{BS}(\undw)$ and $B_w\otimes_{R}\mathbb{R}=\bar{B_w}$.

Let $\underline{w}=s_1s_2\ldots s_\ell$ be a (not necessarily reduced) word of length $\ell$ and $e\in\{0,1\}^{\ell}$ be a $01$-sequence. As explained in \cite[Section 2.4]{EW2}, to a $01$-sequence we can associate a sequence of elements of $\{U0,U1,D0,D1\}$. 
Let $\defect(e)$ be the defect of $e$, i.e. the number of $U0$'s minus the number of $D0$'s of $e$. We define $\Downs(e)$ to be the number of $D$'s (both $D1$'s and $D0$'s) of $e$.
We denote by $\undw^e$ the element $s_1^{e_1}s_2^{e_2}\ldots s_\ell^{e_\ell}$. We have 
\begin{equation}\label{LLdeg}
\defect(e)=\ell(\undw)-\ell(\undw^e)-2\Downs(e).
\end{equation}
For any $k$, $0\leq k\leq \ell$, let $\undw_{\leq k}=s_1s_2\ldots s_k$ and $\undw^e_{\leq k}=s_1^{e_1}s_2^{e_2}\ldots s_k^{e_k}$. We say $x\leq \undw$ if there exists $e$ such that $\undw^e=x$. 
For any element $x\in W$ we denote by $\mathcal{R}(x)$ its right descent set, i.e. $\mathcal{R}=\{s\in S\mid xs<x\}$.

\begin{lemma}\label{cansub}
Let $\undw$ be a word. For any $x\leq \undw$ there exists a unique $01$-sequence $e$ such that $\underline{w}^e=x$ and $e$ has only $U0$'s and $U1$'s. Moreover such $e$ is the unique $01$-sequence of maximal defect with $\undw^e=x$, and $\defect(e)=\ell(\undw)-\ell(x)$.
\end{lemma}
\begin{proof}
We first prove the existence. Let $\underline{w}=s_1\ldots s_\ell$. We start with $x_\ell=x$ and we define recursively, starting with $k=l$ and down to $k=1$,
$$e_k=\begin{cases}
1 & \text{ if }s_k\in \mathcal{R}(x_k)\\
0 & \text{ if }s_k\not \in \mathcal{R}(x_k)
\end{cases},\qquad x_{k-1}=x_k\cdot s_k^{e_k}$$
It follows that $s_{k}\not\in \mathcal{R}(x_{k-1})$ for any $1\leq k\leq \ell$, hence $e$ has only $U1's$ and $U0's$. At any step we have $x_k\leq \undw_{\leq k}$, therefore $x_0=\Iden$ and $\undw^e=x$. 

To show the uniqueness, assume that there are two $01$-sequences $e$ and $f$ with only $U$'s and satisfying $\undw^e=x=\undw^{f}$.
If $e_\ell=f_\ell$ we can conclude that $e=f$ by induction on $\ell$. Otherwise we can assume $e_\ell=1$ and $f_\ell=0$. Now we get $\undw_{\leq \ell -1}^{f}=x$, and $xs_\ell<x$ because the last bit of $e$ is a $U1$. But this means that the last bit of $f$ is a $D0$, hence we get a contradiction.

The last statement follows from \eqref{LLdeg}.
\end{proof}

\begin{definition}
Let $\undw$ be a word and $x\leq \undw$. We call the unique $01$-sequence $e$  without $D$'s such that $\undw^e=x$ the \textit{canonical} sequence for $x$.
\end{definition}

In \cite[Chapter 6]{EW2} Libedinsky's Light Leaves are introduced in the diagrammatic setting. We make use of Elias and Williamson's results.

Let $\undw$ a word and $e$ a $01$-sequence with $\undw^e=x$. The Light Leaf $LL_{\underline{w},e}$ is an element in $\Hom(BS(\underline{w}),BS(\underline{x}))$, for some choice of a reduced expression $\underline{x}$ of $x$. 
For any light leaf $LL_{\undw,e}$, let $\flipLL_{\undw,e}\in \Hom(BS(\underline{x}),BS(\underline{w}))$ be the morphism obtained by flipping the diagram of $LL_{\undw,e}$ upside down. If $\undw^e=\undw^f$ let $\bbLL_{\underline{w},e,f}=\flipLL_{\undw,e}\circ LL_{\undw,f}$. We know from \cite[Theorem 6.11]{EW2} that  the set $\{\bbLL_{\underline{w},e,f}\}_{\undw^e=\undw^f}$ is a basis of $\End(BS(\underline{w}))$ as a right $R$-module.

Let $ll_{\underline{w},e}=\flipLL_{\underline{w},e}(1^\otimes_{\underline{x}})$, where $1^\otimes_{\underline{x}}=1\otimes 1\otimes\ldots\otimes 1\in BS(\underline{x})$. 
We have $\deg(ll_{\undw,e})=-\ell(\undw^e)+\defect(e)$. 
In particular $e$ is a canonical $01$-sequence if and only if $\deg(ll_{\undw,e})+2\ell(\undw^e)=\ell(\undw)$. If there is at least one $D$ in $e$ then the inequality $\deg(ll_{\undw,e})+2\ell(\undw^e)\leq \ell(\undw)-2$ holds.

\begin{lemma}\label{D=0}
Let $\undw$ be a word and $e$ be a $01$-sequence. Then
$$LL_{\undw,e}(1^\otimes_\undw)=
\begin{cases}1^\otimes_{\underline{x}} & \text{ if } e\text{ has only $U$'s}\\0 & \text{ if } e\text{ has (at least) one }D.\end{cases}$$
\end{lemma}
\begin{proof}
The statement  easily follows from the definitions when $e$ has only $U$'s. By induction on $\ell(\undw)$ we can assume that $e$ has only one $D$ at the right end. Then $LL_{\undw,e}$ looks like
 
 $$ 
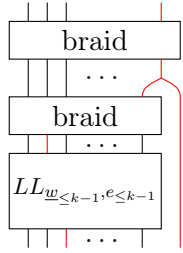
\begin{tikzpicture}[scale=0.5]thick, 
\draw (-0.5,0) rectangle (3.5,2); 
\draw (-0.5,2.5) rectangle (3.5,3.5);
\draw (-0.5,4.5) rectangle (4,5.5);
\node[scale=0.8] at (1.5,1) {$LL_{\undw_{\leq k-1}, e_{\leq k-1}}$}; % first box
\node at (1.5,3) {braid};
\node at (1.75,5) {braid};
\foreach \x in {0,0.5,3} \draw(\x,-0.5) -- (\x,0);
\draw[red] (1,-0.5) -- (1,0);
\foreach \x in {0,1,3} \draw(\x,2) -- (\x,2.5);
\draw[red] (0.5,2) -- (0.5,2.5);
\foreach \x in {0,0.5,1} \draw(\x,3.5) -- (\x,4.5);
\foreach \x in {0,0.5,1} \draw(\x,5.5) -- (\x,6);
\draw[red] (3.5,5.5) to (3.5,6);
\node at (2,-0.3) {\dots};
\node at (2,2.2) {\dots};
\node at (2, 4) {\dots};
\draw[red] (4,-0.5) to (4,3.5) to [out=90, in=-30] (3.5, 4) to [out=-150, in=90] (3,3.5) ;
\draw[red] (3.5,4) to (3.5,4.5);
\end{tikzpicture} $$

The box labelled by ``braid'' contains only $2m_{st}$-valent vertices. By induction $$\left(LL_{\undw_{\leq k-1}, e_{\leq k-1}}\otimes \Iden_{B_{s_{\ell(\undw)}}}\right)\left(1_{\undw}^\otimes\right)=1_{\underline{x}}^\otimes.$$

Notice that every $2m_{st}$-valent vertex preserves $1\otimes 1\otimes \ldots 1$. It follows from the definition that a trivalent vertex applied to $1\otimes 1\otimes 1$ returns $0$, thus $LL_{\undw,e}(1^\otimes_\undw)=0$.
\end{proof}

\begin{cor}
Let $\undw$ be a word.
The set $\{ll_{\underline{w},e}\}$ with $e\in\{0,1\}^{\ell(\undw)}$ is a basis of $BS(\underline{w})$ as a right $R$-module.
\end{cor}
\begin{proof}

Let $\undw=s_1s_2\ldots s_\ell$ with $\ell=\ell(w)$.
We first show that the span of $\{\phi(1^\otimes_{\undw})\}$ with $\phi\in\End(BS(\underline{w}))$ generates $BS(\underline{w})$. Then for example one could consider the morphisms $\phi_e$, for any $e\in\{0,1\}^\ell$, defined by:
 $$ \begin{tikzpicture}[scale=0.5]thick, 
\foreach \x in {0,1.5,3,4.5,8,14,19,21} {
\draw(\x,1.25) -- (\x,0); 
\draw(\x,2.25) -- (\x,3.5);
\node[circle,fill,draw,inner sep=0mm,minimum size=1mm] at (\x,1.25) {};
\node[circle,fill,draw,inner sep=0mm,minimum size=1mm] at (\x,2.25) {};}
\foreach \x in {1,2,3,4} \node at (\x*1.5-1,3.5) {$e_\x$};
\node at (8.5,3.5) {$e_\ell$};
\node at (14.5,3.5) {$0$};
\node at (19.5,3.5) {$1$};
\draw(16,0) -- (16,3.5);
\node at (6.25,1.75) {$\cdots$};
\node at (11,1.75) {where};
\node at (17.5,1.75) {and};
\node at (15,1.75) {$:=$};
\node at (20,1.75) {$:=$};
\node at (-2,1.75) {$\phi_e:=$};
\end{tikzpicture}
$$
We have $\phi_e(1^\otimes_{\undw})=\alpha_{s_1}^{e_1}\otimes \alpha_{s_2}^{e_2}\otimes\ldots \otimes \alpha_{s_\ell}^{e_\ell}$. Since the set $\{\phi_e(1^\otimes_{\undw})\}_{e\in \{0,1\}^\ell}$ is a basis of $BS(\undw)$ as a right $R$-module, the claim follows.

Then clearly also the span of all the ${\bbLL_{\underline{w},e,f}(1^\otimes_\undw)}$ with $\undw^e=\undw^f$ generates $BS(\undw)$.
Applying Lemma \ref{D=0} we see that ${\bbLL_{\underline{w},e,f}(1^\otimes_\undw)}=ll_{\undw,e}$ if $f$ is canonical and $0$ otherwise. It follows that
$\{ll_{\underline{w},e}\}_{e\in\{0,1\}^\ell}$ spans $BS(\undw)$. Since the rank of $BS(\undw)$ as a right $R$-module is $2^{\ell(w)}$ the thesis follows, cf. \cite[Theorem 2.4]{Mat}.
\end{proof}

\begin{remark}
The result of this section are, at least in my knowledge, still unpublished. However Geordie Williamson and Ben Elias explained canonical subexpression and how to construct the basis $\{ll_{\undw,e}\}$ in a master class at the QGM in Aarhus already in 2013. Videos and notes of the lectures are available at \url{http://qgm.au.dk/video/mc/soergelkl/}.
\end{remark}

\subsection{The ``homology'' submodule of an indecomposable Soergel module}

Recall from \cite[Section 3.5]{EW1} that for any Soergel bimodule $B$ we have 
$$\Gamma_{\leq x}B/\Gamma_{< x}B\cong \nabla_x^{\oplus h_x(B)}\text{ with }h_x(B)\in \mathbb{Z}[v,v^{-1}]$$ where $\nabla_x=R_x[\ell(x)]$ is a shift of the standard bimodule $R_x$ and $v$ denotes the degree shift. In particular, if $BS(\undw)$ is a Bott-Samelson bimodule then $h_x(BS(\undw))=\sum_{e\colon \undw^e=x}v^{\defect(e)}$, while if $B_w$ is an indecomposable bimodule, then $h_x(B_w)$ is equal to the  polynomial $h_{x,w}(v)$. The polynomials $h_{x,w}(v)$ are related to the usual Kazhdan-Lusztig polynomials via the formula $h_{x,w}(v)=v^{\ell(w)-\ell(x)}p_{x,w}(v^{-2})$. In particular $h_{x,w}\in \mathbb{Z}[v]$ and $h_{x,w}=v^{\ell(w)-\ell(x)}+$``lower terms,'' for any $x\leq w$.

The basis $\{ll_{\undw,e}\}$ is compatible both with the filtration support and with the degree grading of $BS(\undw)$. In other words, for any $x$ and any $k\in \mathbb{Z}_{\geq 0}$,  the set $\{ll_{\undw,e} \mid \undw^e=x, \defect(e)=k\}$ induces a basis on the summand $\nabla_x[k]^{\oplus c_k}\cug \Gamma_{\leq x}BS(\undw)/\Gamma_{< x}BS(\undw)$, where $c_k$ is the coefficient of $v^k$ in $h_x(BS(\undw))$.

Let's consider the following right $R$-submodules of $BS(\undw)$:
$$C_w=\sum_{e\text{ canonical}} ll_{\undw,e}R\qquad\text{ and }\qquad D_w=\sum_{e\text{ not canonical}}ll_{\undw,e}R.$$
In general $C_w$ it is not a left $R$-module.

\begin{lemma}
Let $D_w$ as above. Then $D_w$ is a $R$-subbimodule of $BS(\undw)$.
\end{lemma}
\begin{proof}
It suffices to show that, for any non-canonical $e$ and for any $f\in R$, we have $f\cdot ll_{\undw,e}=\sum_{i} ll_{\undw,e_i}g_i$, with  $e_i$ not canonical and $g_i\in R$. Since $R$ is generated in degree $2$ we can assume $f$ to be homogeneous of degree $2$.

Let $x=\undw^e$. The element $f\cdot ll_{\undw,e}$ is contained in $\Gamma_{\leq x}(BS(\undw))$. Using repeatedly the nil-Hecke relation \cite[(5.2)]{EW2}
on the bottom of the diagram we see that 
\begin{equation}\label{nilHecke}
f\cdot ll_{\undw,e}=ll_{\undw,e}\cdot x^{-1}(f) +\Theta,\end{equation}
with $\Theta\in \Gamma_{<x}(BS(\undw))$. 

Therefore we can write $\Theta=\sum_i ll_{\undw,f_i}h_i$, with $h_i\in R$ and $\undw^{f_i}<x$. Furthermore, since the equation \eqref{nilHecke} is homogeneous, if $h_i\neq 0$ we have 
$\deg(h_i)+\deg(ll_{\undw,f_i})=\deg(f)+\deg(ll_{\undw,e})=\deg(ll_{\undw,e})+2$ for all $i$, whence 
$$\deg(ll_{\undw,f_i})\leq \deg(ll_{\undw,e})+2\leq \ell(\undw)-2\ell(x)<\ell(\undw)-2\ell(\undw^{f_i})$$
and $f_i$ must be not canonical.
\end{proof}

Let now $\undw$ be a reduced word. Fix a decomposition of $BS(\undw)$ into indecomposable bimodules and let $E_w\in \End(BS(\undw))$ be the primitive idempotent corresponding to $B_w$, i.e. $BS(\undw)=\Ker(E_w)\oplus \ima(E_w)$ and $\ima(E_w)\cong B_w$. Since, for any $x$, the map 
$$\Gamma_{\leq x}BS(\undw)/\Gamma_{< x}BS(\undw)\raw \Gamma_{\leq x}B_w/\Gamma_{< x}B_w$$
induced by $E_w$ is surjective, it follows that the projection of the set $\{E_w(ll_{\undw,e}) \mid \undw^e=x, \defect(ll_{\undw,e})=k\}$ spans the summand $\nabla_x[k]^{\oplus c_k(h_{x,w})}$  of $\Gamma_{\leq x}B_w/\Gamma_{< x}B_w$, where $c_k(h_{x,w})$ is the coefficient of $v^k$ in $h_{x,w}$.

In particular, because of Lemma \ref{cansub}, 
for any $x\leq w$ the summand $\nabla_x[\ell(w)-\ell(x)]\cug \Gamma_{\leq x}BS(\undw)/\Gamma_{< x}BS(\undw)$ is spanned by $ll_{\undw,e}$, where $e$ is the canonical sequence for $x$.
Moreover, we have $h_{x,w}=v^{\ell(w)-\ell(v)}+$``lower terms,'' hence the summand  $\nabla_x[\ell(w)-\ell(x)]\cug \Gamma_{\leq x}B_w/\Gamma_{< x}B_w$ has as a basis the projection of $\{E_w(ll_{\undw,e})\}$.
Therefore, the map
\begin{equation}\label{mapk}
\Gamma_{\leq x}BS(\undw)/\Gamma_{< x}BS(\undw)\otimes \mathbb{R}\raw \Gamma_{\leq x}B_w/\Gamma_{< x}B_w\otimes \mathbb{R}
\end{equation}
is an isomorphism in degree $v^{\ell(w)-2\ell(x)}$.

Let $\bar{C_w}=C_w\otimes_R \mathbb{R}$, $\bar{D_w}=D_w\otimes_R \mathbb{R}$ and let us denote by $\bar{E_w}: \bar{BS}(\undw)\raw \bar{B_w}$ the induced morphism of left $R$-modules. For any $e$, let $\bar{ll}_{\undw,e}$ denote the projection of $ll_{\undw,e}$ to $\bar {BS}(\undw)$.

\begin{lemma}\label{kerDw}
The kernel of $\bar{E_w}$ is contained in $\bar{D_w}$. 
\end{lemma}
\begin{proof}
Let $\sum_i \bar{ll}_{\undw,e_i}g_i\in \Ker \bar{E_w}$, with $g_i\in \mathbb{R}$. Since $\bar{E_w}$ is homogeneous we can assume the sum to be homogeneous.
Assume that a canonical sequence $e_j$ appears in the sum with $g_j\neq 0$. Then $\undw^{e_j}\neq \undw^{e_i}$ for any $i\neq j$ with $g_i\neq 0$ and, in addition, $x:=\undw^{e_j}$ must be of maximal length among $X:=\{\undw^{e_i} \mid g_i \neq 0\}$. 

We can also choose a refinement of the Bruhat order into a total order of $W$ such that $x$ is maximal inside $X$. We label the elements of $W$ as $w_1<w_2<\ldots$ in order.

For an integer $k\geq 1$ let's denote by $\Gamma_{\leq k}B$ the submodule of elements supported on $\{w_1,\ldots, w_k\}$. Then by Soergel hin-und-her Lemma \cite[Lemma 6.3]{S4} we have for any Soergel bimodule $B$,
$$\Gamma_{\leq w_k}B/\Gamma_{< w_k}B\cong \Gamma_{\leq k}B/\Gamma_{\leq k-1}B.$$

Let $h$ be the index of $x$, i.e. $x=w_h$.
We have $\sum ll_{\undw,e_i}g_i\in \Gamma_{\leq h}BS(\undw)$ and projects to $ll_{\undw,e_j}g_j\in \Gamma_{\leq h}BS(\undw)/\Gamma_{\leq h-1}BS(\undw)$. But the map
\begin{equation}
\Gamma_{\leq h}BS(\undw)/\Gamma_{\leq h-1}BS(\undw)\otimes \mathbb{R}\raw \Gamma_{\leq h}B_w/\Gamma_{\leq  h-1}B_w\otimes \mathbb{R}
\end{equation}
is an isomorphism in degree $v^{\ell(w)-2\ell(x)}$.
Hence $\sum \bar{ll}_{\undw,e_i}g_i$, or equivalently $\bar{ll}_{\undw,e_j}g_j$, is sent to $0$ if and only if $g_j=0$. We obtain a contradiction, whence $\sum \bar{ll}_{\undw,e_i}g_i\in \bar{D_w}$. 
\end{proof}

%\begin{remark}
%A refinement of the proof of Lemma \ref{kerDw} shows that the kernel of $E_w$ lies inside $D_w$. We omit it, since it is unnecessary for our purposes.
%\end{remark}

It follows that $\bar{B_w}=\bar{E_w}(\bar{C_w})\oplus \bar{E_w}(\bar{D_w})$ as $\mathbb{R}$-vector spaces. Moreover, $\bar{E_w}(\bar{D_w})$ is a $R$-submodule of $\bar{B_w}$ and the restriction of $\bar{E_w}$ to $\bar{C_w}$ is injective. We now have all the tools to generalize Corollary \ref{corPD} to the setting of a general Coxeter group.

\begin{cor}
For any $w \in W$ the following are equivalent:
\begin{enumerate}[i)]
 \item $\bar{E_w}(\bar{C_w})\cong\bar{B_w}$.
 \item $\#\{v\in W\mid v\leq w\text{ and }\ell(v)=k\}=\#\{v\in W\mid v\leq w\text{ and }\ell(v)=\ell(w)-k\}$ for any $k\in \bbZ$.
 \item All the Kazhdan-Lusztig polynomials $p_{v,w}$ are trivial.
\end{enumerate}
\end{cor}
\begin{proof}
From Lemma \ref{kerDw} we have 
$$\dim (\bar{E_w}(\bar{C_w}))^k=\dim(\bar{C_w})^k=\#\{v\in W \mid v\leq w\text{ and }2\ell(v)=\ell(w)-k\}.$$ Notice that ii) holds if and only if  we have $\dim (\bar{E_w}(\bar{C_w}))^k=\dim (\bar{E_w}(\bar {C}_w))^{-k}$ for any $k\in \mathbb{Z}$, hence if and only if $\dim (\bar{E_w}(\bar{D_w}))^k=\dim (\bar{E_w}(\bar{D_w}))^{-k}$ for any $k\in \mathbb{Z}$. 

If ii) holds, then  we can apply Corollary \ref{betti} to the $R$-submodule $\bar{E_w}(\bar{D_w})\cug \bar{B_w}$. Since $\bar{B_w}$ is indecomposable it follows that $\bar{E_w}(\bar{D_w})=0$. Hence ii) implies i).

 The rest of the proof continues just as in Corollary \ref{corPD}, where $IH_w$ is replaced by $\bar{B_w}$ and $H_w[\ell(w)]$ by $\bar{E_w}(\bar{C_w})$.
\end{proof}

\begin{remark}
One could also define $\tilde{H}_w:=\bar{E_w}(\bar{D_w})^\perp$, where the orthogonal is taken with respect to the intersection form of $\bar{B_w}$, and check that $\tilde{H}_w$ coincides with $H_w$ if $W$ is the Weyl group of some reductive group $G$.
\end{remark}

\Address


\begin{thebibliography}{CE+}

\bibitem[BBD]{BBD}
A.~A. Be{\u\i}linson, J.~Bernstein, and P.~Deligne,
{\it Faisceaux pervers},
Analysis and topology on singular spaces, vol.~100 of  Ast\'erisque, 1982, 5--171.

\bibitem[BGG]{BGG}
I.~N. Bern{\v{s}}te{\u\i}n, I.~M. Gel{\cprime}fand, and S.~I. Gel{\cprime}fand.
{\it Schubert cells, and the cohomology of the spaces {$G/P$}},
 Uspehi Mat. Nauk, {\bf 28} (1973), 3--26.

\bibitem[B]{Bou}
N.~Bourbaki.
{\it \'{E}l\'ements de math\'ematique, {G}roupes et
  alg\`ebres de {L}ie. {C}hapitre {IV-VI}},
  Paris, 1968.

\bibitem[Ca]{Ca}
J.~B. Carrell.
{\it The {B}ruhat graph of a {C}oxeter group, a conjecture of {D}eodhar,
  and rational smoothness of {S}chubert varieties},
in Algebraic groups and their generalizations: classical methods
  ({U}niversity {P}ark, {PA}, 1991),  vol.~56 of  Proc. Sympos. Pure
  Math., 1994, 53--61.

\bibitem[CE+]{CaZe}
E.~Cattani, F.~{El Zein}, P.~Griffiths, and L.~Trang.
{\it Hodge Theory (MN-49)}.
Princeton University Press, 2014.

\bibitem[dCM]{dCM2}
M.~A.~{de Cataldo} and L.~Migliorini.
{\it Hodge-theoretic aspects of the decomposition theorem},
in Algebraic geometry---{S}eattle 2005. {P}art 2, vol.~80 of
   Proc. Sympos. Pure Math., 2009, 489--504.

\bibitem[D]{Dy}
E.~B. Dynkin.
{\it Maximal subgroups of the classical groups},
Trudy Moskov. Mat. Ob\v s\v c., 1952, 39--166.

\bibitem[EW1]{EW1}
B.~Elias and G.~Williamson.
{\it The {H}odge theory of {S}oergel bimodules},
Ann. of Math., {\bf 180} (2014), 1089--1136.

\bibitem[EW2]{EW2}
B.~Elias and G.~Williamson.
{\it Soergel calculus},
Represent. Theory, {\bf 20} (2016), 295--374.

\bibitem[G]{Gi}
V. Ginsburg.
{\it Perverse sheaves and {${\bf C}^*$}-actions},
J. Amer. Math. Soc.
, {\bf 4} (1991), 483--490.

\bibitem[LL]{LL}
E.~Looijenga and V.~A.~Lunts.
{\it A {L}ie algebra attached to a projective variety},
 Invent. Math., {\bf 129} (1997), 361--412.

\bibitem[M]{Mat}
H. Matsumura.
{\it Commutative ring theory},  vol.~8 of Cambridge Studies in
  Advanced Mathematics,
Cambridge University Press, 1986.
Translated from the Japanese by M. Reid.

\bibitem[S1]{S1}
W. Soergel.
{\it Kategorie {$\scr O$}, perverse {G}arben und {M}oduln \"uber den
  {K}oinvarianten zur {W}eylgruppe},
J. Amer. Math. Soc., {\bf 3} (1990), 421--445.

\bibitem[S2]{S5}
W. Soergel.
{\it On the relation between intersection cohomology and representation
  theory in positive characteristic}, J. Pure Appl. Algebra, {\bf 152} (2000), 311--335.

\bibitem[S3]{S4}
W. Soergel.
{\it Kazhdan-{L}usztig-{P}olynome und unzerlegbare {B}imoduln \"uber
  {P}olynomringen},
 J. Inst. Math. Jussieu, {\bf 6} (2007), 501--525.

\bibitem[W]{W4}
G. Williamson.
{\it Singular {S}oergel bimodules}, 
Int. Math. Res. Not., {\bf 20} (2011), 556--587.





\end{thebibliography}
\end{document}